\documentclass[11pt]{amsart}
\usepackage{enumerate,amssymb}
\usepackage{amsopn,amsfonts,amsmath}
\usepackage[dvips]{graphicx}
\usepackage{float}
\usepackage{hyperref}
\hypersetup{colorlinks=true, pdfborderstyle={/S/U/W 1},
urlbordercolor={0 0 1}, urlcolor=blue}
\newtheorem{thm}{Theorem}[section]
\newtheorem{cor}{Corollary}[section]
\newtheorem{lem}{Lemma}[section]
\newtheorem{prop}{Proposition}[section]
\newtheorem{ass}{Assumption}[section]

\theoremstyle{definition}
\newtheorem{defn}{Definition}[section]

\newtheorem{rem}{Remark}[section]
\newtheorem*{notation}{Notation}

\numberwithin{equation}{section}

\newcommand{\textupandbold}[1]{\textup{\textbf{#1}}}

\newcommand{\E}{\mathbb{E}}
\newcommand{\R}{\mathbb{R}}
\newcommand{\I}{\mathbb{I}}

\allowdisplaybreaks

\def\xE{{\mathbb E}}

\def\xR{{\mathbb R}}
\def\xZ{{\mathbb Z}}
\def\xN{{\mathbb N}}
\def\xP{{\mathbb P}}

\def\supp{{\rm supp}}
\def\xS1{{\mathbb S}^1}
\def\xSd{{\mathbb S}^{d-1}}

\def\spn{{\rm span}}

\def\xdif{{\rm d}}

\def\xLtwo{{{\rm L}^2}}
\def\xLone{{{\rm L}^1}}

\def\xLn{{\rm L}}
\def\xHn{{\rm H}}
\def\xWn{{\rm W}}

\def\xLinfty{{\rm L}^{\infty}}

\def\xker{{\rm ker}}

\begin{document}
\title[]{Nonparametric Estimation in Random
Coefficients Binary Choice Models}

\author[Gautier]{Eric Gautier}
\address{CREST (ENSAE),
3 avenue Pierre Larousse, 92245 Malakoff Cedex, France.}
\email{eric.gautier@ensae-paristech.fr}

\author[Kitamura]{Yuichi Kitamura}
\address{Cowles Foundation for Research in Economics,
Yale University, New Haven, CT-06520.}
\email{yuichi.kitamura@yale.edu}

\date{This Version: August 31, 2011.}

\thanks{We thank Whitney Newey and two anonymous referees for
comments that greatly improved this paper. We also thank seminar
participants at Chicago, CREST, Harvard/MIT,
the Henri Poincar\'e Institute, Hitotsubashi, LSE, Mannheim,
Minnesota, Northwestern, NYU, Paris 6, Princeton, Rochester, Simon
Fraser, Tilburg, Toulouse 1 University, UBC, UCL, UCLA, UCSD, the
Tinbergen Institute and the University of Tokyo, and participants of
the 2008 Cowles summer econometrics conference, EEA/ESEM, FEMES,
Journ\'ees STAR, and SETA and 2009 CIRM Rencontres de Statistiques
Math\'ematiques for helpful comments.  Yuhan Fang and Xiaoxia Xi
provided excellent research assistance. Kitamura acknowledges
financial support from the National Science Foundation via grants
SES-0241770, SES-0551271 and SES-0851759.  Gautier is grateful for support from
the Cowles Foundation as this research was initiated during his
visit as a  postdoctoral associate.}

\begin{abstract}
This paper considers random coefficients binary choice models. The
main goal is to estimate the density of the random coefficients
nonparametrically.  This is an ill-posed inverse problem
characterized by an integral transform.  A new density estimator for
the random coefficients is developed, utilizing Fourier-Laplace
series on spheres.  This approach offers a clear insight on the
identification problem.  More importantly, it leads to a closed form
estimator formula that yields a simple plug-in procedure requiring
no numerical optimization.  The new estimator, therefore, is easy to
implement in empirical applications, while being flexible about the
treatment of unobserved heterogeneity.  Extensions including
treatments of non-random coefficients and models with endogeneity
are discussed.
\end{abstract}

\maketitle

\section{Introduction}\label{sec:intro}
Consider a binary choice model
\begin{equation}\label{emod}
Y=\I\left\{X'\beta\geq0\right\}
\end{equation}
where $\I$ denotes the indicator function and $X$ is a $d$-vector of
covariates.  We assume that the first element of $X$ is 1, therefore the
vector $X$ is of the form $X=(1,\tilde{X}')'$.  The vector
$\beta$ is random. The random element $(Y,\tilde{X},\beta)$ is
defined on some probability space $(\Omega,\mathcal{F},\mathbb{P})$,
and $(y_i,\tilde{x}_i,\beta_i), i = 1,...,N$ denote its
realizations. The econometrician observes $(y_i,\tilde{x}_i), i =
1,...,N$, but $\beta_i, i = 1,...,N$ remain unobserved.  The vectors
$\tilde{X}$ and $\beta$ correspond to observed and unobserved
heterogeneity across agents, respectively.  Note that the first
element of $\beta$ in this formulation absorbs the usual scalar
stochastic shock term as well as a constant in a standard binary
choice model with non-random coefficients.  This formulation is used in
Ichimura and Thompson (1998), and is convenient for the subsequent
development in this paper.  Our basic model maintains exogeneity of
the covariates $\tilde X$:
\begin{ass}\label{ass1}
$\beta$ is independent of $\tilde{X}$,
\end{ass}
\noindent Section \ref{sec:endogeneity} considers ways to relax
this assumption.  Under \eqref{emod} and Assumption \ref{ass1},
the choice probability function is given by
\begin{eqnarray}
r(x)&=\mathbb{P}(Y=1|X=x)\label{e1}
\\
&=\E_\beta[\mathbb{I}\left\{x'\beta>0\right\}].
\nonumber
\end{eqnarray}
Discrete choice models with random coefficients are useful
in applied research since it is often crucial to incorporate
unobserved heterogeneity in modeling the choice behavior of individuals.  There is
a vast and active literature on this topic.  Recent contributions include
Briesch, Chintagunta and Matzkin (1996), Brownstone and Train (1999), Chesher
and Santos Silva (2002), Hess, Bolduc and Polak (2005), Harding and Hausman (2006),
Athey and Imbens (2007), Bajari, Fox and Ryan (2007) and Train (2003).  A
common approach in estimating random coefficient discrete choice
models is to impose parametric distributional assumptions. A leading example is
the mixed Logit model, which is discussed in details by Train
(2003).  If one does not impose a parametric distributional
assumption, the distribution of $\beta$ itself is the structural
parameter of interest.  The goal for the econometrician is then to
recover it nonparametrically from the information about
$r(x)$ obtained from the data.

Nonparametric treatments for unobserved heterogeneity distributions
have been considered in the literature for other models. Heckman and
Singer (1984) study the issue of unobserved heterogeneity
distributions in duration models and propose a treatment by a
nonparametric maximum likelihood estimator (NPMLE). Elbers and
Ridder (1982) also develop some identification results in such
models.  Beran and Hall (1992) and Hoderlein et al. (2007) discuss
nonparametric estimation of random coefficients linear regression
models.  Despite the tremendous importance of random coefficient
discrete choice models, as exemplified in the above references,
nonparametrics in these models is relatively underdeveloped.  In
their important paper, Ichimura and Thompson (1998) propose an NPMLE
for the CDF of $\beta$.  They present sufficient conditions for
identification and prove the consistency of the NPMLE.  The NPMLE
requires high dimensional numerical maximization and can be
computationally intensive even for a moderate sample size.  Berry
and Haile (2008) explore nonparametric identification problems in a
random coefficients multinomial choice model that often arises in
empirical IO.

This paper considers nonparametric estimation of the random
coefficients distribution, using a novel approach that shares some
similarities with standard deconvolution techniques.  This allows us
to reconsider the identifiability of the model and obtain a
constructive identification result.  Moreover, we  develop a simple
plug-in estimator for the density of $\beta$ that requires no
numerical optimization or integration. It is easy to implement in
empirical applications, while being flexible about the treatment of
unobserved heterogeneity.

Since the scale of $\beta$ is not identified in the binary choice
model, we normalize it so that $\beta$ is a vector of Euclidean norm
1 in $\xR^d$.  The vector $\beta$ then belongs to the $d-1$
dimensional sphere $\xSd$. This is not a restriction as long as the
probability that $\beta$ is equal to 0 is 0. Also, since only the
angle between $X$ and $\beta$ matters in the binary decision
$\I\{X'\beta \geq 0\}$, we can replace $X$ by $X/\|X\|$ without any
loss of information.  We therefore assume that $X$ is on the sphere
$\xSd$ as well in the subsequent analysis.  Results from the
directional data literature are thus relevant to our analysis.  We
aim to recover the joint
probability density function $f_{\beta}$ of
$\beta$ with respect to the uniform spherical measure $\sigma$ over
$\xSd$ from the random sample $(y_1,x_1),\hdots,(y_N,x_N)$ of
$(Y,X)$.

The problem considered here is a linear ill-posed inverse problem.
We can write
\begin{equation}\label{e2}
r(x)=\int_{b\in\xSd}\I\left\{x'b\geq0\right\}f_{\beta}(b)d\sigma(b)
=\int_{H(x)}f_{\beta}(b)d\sigma(b)
:=\mathcal{H}\left(f_{\beta}\right)(x)
\end{equation}
where the set $H(x)$ is the hemisphere $\left\{b:x'b\geq0\right\}$.
The mapping $\mathcal{H}$ is called the hemispherical
transformation. Inversion of this mapping was first studied by Funk
(1916) and later by Rubin (1999). Groemer (1996) also discusses some
of its properties. $\mathcal{H}$ is not injective without further
restrictions and conditions need to be imposed to ensure
identification of $f_\beta$ from $r$. Even under a set of
assumptions that guarantee identification, however, the inverse of
$\mathcal{H}$ is not a continuous mapping, making the problem
ill-posed. To see this, suppose we restrict $f_\beta$ to be in
$\xLtwo(\xSd)$.  Since the kernel of $\mathcal H$ is square
integrable by compactness of the sphere,  it is Hilbert-Schmidt and
thus compact.  Therefore if the inverse of $\mathcal H$ were
continuous, ${\mathcal H}^{-1} \mathcal H$ would map the closed unit
ball in $\xLtwo(\xSd)$ to a compact set. But the Riesz theorem
states that the unit ball is relatively compact if and only if the
vector space has finite dimension.  The fact that $\xLtwo(\xSd)$ is
an infinite dimensional space contradicts this.  Therefore the
inverse of $\mathcal H$ cannot be continuous. In order to overcome
this problem, we use a one parameter family of regularized inverses
that are continuous and converge to the inverse when the parameter
goes to infinity. This is a common approach to ill-posed inverse
problems in statistics (see, e.g. Carrasco et al., 2007).

Due to the particular form of its kernel that involves the scalar
product $x'b$, the operator $\mathcal H$ is an analogue of
convolution in $\xR^d$, as illustrated in a simple example in
Section \ref{sec:toy} of Supplemental Appendix.  This analogy
provides a clear insight into the identification issue.  In
particular, our problem is closely related to the so-called boxcar
deconvolution (see, e.g. Groeneboom and Jongbloed (2003) and
Johnstone and Raimondo (2004)), where identifiability is often a
significant problem. The connection with deconvolution is also
useful in deriving an estimator based on a series expansion on the
Fourier basis on $\mathbb S^1$ or its extension to higher
dimensional spheres called Fourier-Laplace series. These bases are
defined via the Laplacian on the sphere, and they diagonalize the
operator $\mathcal H$ on $\xLtwo\left(\xSd\right)$.  Such techniques
are used in Healy and Kim (1996) for nonparametric empirical Bayes
estimation in the case of the sphere $\mathbb{S}^2$.  The kernel of
the integral operator $\mathcal H$, however, does not satisfy the
assumptions made by Healy and Kim.  Unlike Healy and Kim (1996), we
make use of so-called ``condensed'' harmonic expansions.  The
approach replaces a full expansion on a Fourier-Laplace basis by an
expansion in terms of the projections on the finite dimensional
eigenspaces of the Laplacian on the sphere. This is useful since an
explicit expression of the kernel of the projector is available.  It
enables us to work in any dimension and does not require a
parametrization by hyperspherical coordinates nor the actual
knowledge of an orthonormal basis.  This approach, to the best of
our knowledge, appears to be new in the econometrics literature.

The paper is organized as follows. Section \ref{sec:preliminaries}
provides a practical guide for our procedure, which is easy to
implement.    Section \ref{sec:deconvolution} deals with
identification while introducing basic notions used throughout the
paper.  We derive the convergence rates of the estimators in all the
$\xLn^q$ spaces for $q \in [1,\infty]$ and also prove a pointwise
CLT in Section \ref{sec:ParticularEstimatefBeta}.  Some extensions,
such as estimation of marginals, treatments of models with
non-random coefficients, and the case with endogenous regressors are
presented in Section \ref{sec:extensions}.  Simulation results are
reported in Section \ref{sec:simulation}.  Section
\ref{sec:conclusion} concludes.  Supplemental Appendix presents
analysis of a toy model, technical tools used in the main text,
estimators for choice probabilities that are used to construct our
density estimators, and the proofs of the main results.

\section{A Brief Guide for Practical Implementation}\label{sec:preliminaries}
This section presents our basic estimation procedure when a random
sample $\{(y_i,\tilde x_i)\}$ generated from the model \eqref{emod}
is available. As noted in Section \ref{sec:intro}, normalize
covariates data and define $x_i = (1,\tilde x_i')/\|(1,\tilde
x_i')\| \in \mathbb S^{d-1}, i = 1,...,N$.   To estimate the joint
density of the random vector $\beta$, use the following formula:
\begin{equation}\label{mainestimator} \hat f_\beta(b) =
\max\left(\frac 2 {|\mathbb S^{d-1}|} \sum_{p=0}^{T_N - 1}
\frac{\chi(2p+1,2T_N)h(2p+1,d)}{\lambda(2p+1,d)
C_{2p+1}^{\nu(d)}(1)} \left( \frac 1 N \sum_{i=1}^N \frac {(2y_i-1)
C_{2p+1}^{\nu(d)}(x_i'b)}{\max\left(\hat{f}_X(x_i), m_N\right)}
\right),0\right).
\end{equation}
The factors $|\xSd|$, $\chi$, $h$ and $\lambda$ are constants that
do not depend on data and trivial to compute.  The surface area
$\left|\xSd\right|$ of $\xSd$ is given by
$\left|\xSd\right|=\frac{2\pi^{d/2}}{\Gamma (d/2)}$ where $\Gamma$
denotes the Gamma function.  The constants $h$, $\nu$ and $\lambda$
are obtained via the numerical formulas $h(n,d) =
\frac{(2n+d-2)(n+d-2)!}{n!(d-2)!(n+d-2)}$, $\nu(d) = (d-2)/2$ and
$\lambda(2p+1,d)=\frac{(-1)^{p}|\mathbb{S}^{d-2}|1\cdot3\cdots(2p-1)}
{(d-1)(d+1)\cdots(d+2p-1)}$, respectively.  The function $\chi$ is
defined on $\mathbb N \times \mathbb N$ and used for smoothing. This
is to be chosen by the user: see Proposition \ref{pkernel} as well
as the numerical example reported in Section \ref{sec:simulation}
for examples of $\chi$.  The truncation parameter $T_N$ needs to be
chosen so that it grows with the sample size with a sufficiently
slow rate.  The trimming factor $m_N$ is also user-defined, and it
is chosen so that it goes to zero as the sample size increases.  The
notation $C_n^\nu(\cdot)$ signifies the Gegenbauer
polynomial\footnote{The Gegenbauer polynomials are given by
$$
C_{n}^{\nu}(t)=\sum_{l=0}^{[n/2]}\frac{(-1)^{l}(\nu)_{n-l}}{l!(n-2l)!}(2t)^{n-2l},
\quad \nu > -1/2, n \in \mathbb N
$$
where $(a)_0=1$ and for $n$ in $\xN\setminus\{0\}$,
$(a)_n=a(a+1)\cdots(a+n-1)=\Gamma(a+n)/\Gamma(a)$.  See Section
\ref{sec:premresults} for further properties of the Gegenbauer
polynomials.}; They, for example, correspond to the Chebychev
polynomials of the first kind in the case of one random slope (i.e.
the case with $d=2$)\footnote{When $d=2$, the following relations
can be used in \eqref{mainestimator} and \eqref{fxest}
\begin{align*}
&\forall p\ge0,\ \frac 1 {|\mathbb S^{d-1}|}
\frac{h(2p+1,2)C_{2p+1}^{0}(x_i'b)}{\lambda(2p+1,2)
C_{2p+1}^{0}(1)}=\frac{(-1)^p(2p+1)}{\pi}\cos\left((2p+1)\arccos(x_i'b)\right),\\
&\forall n\ge0,\ \frac 1 {|\mathbb S^{d-1}|}
\frac{h(n,2)C_{n}^{0}(x_i'b)}{C_{n}^{0}(1)}=\frac{1}{\pi}\cos\left(n\arccos(x_i'b)\right).
\end{align*}}.  The only remaining factor which needs to be calculated in the above formula
is the nonparametric density estimator $\hat f_X$ for $f_X$ on $\xSd$.  For example, the
following nonparametric estimator can be used:
\begin{equation}\label{fxest}
\hat f_X(x) = \max \left(\frac 1 {|\mathbb S^{d-1}|}
\sum_{n=0}^{T_N'} \frac{\chi(n,T_N')h(n,d)}{
C_{n}^{\nu(d)}(1)} \left( \frac 1 N \sum_{i=1}^N
C_{n}^{\nu(d)}(x_i'x) \right ), 0\right)
\end{equation}
where $T_N'$ is an another truncation parameter, playing a role similar to $T_N$.

Our estimator $\hat f_\beta$ requires neither numerical integration
nor optimization.  This is a clear advantage over existing
estimators for random coefficient binary choice models, including
many parametric estimators.   This is our main proposal, on which
the rest of the paper focuses.  In  Section
\ref{sec:ParticularEstimatefBeta} we explain how the formula
\eqref{mainestimator} is derived, and investigate its asymptotic
properties.

\section{Identification Analysis}\label{sec:deconvolution}
In this section we address the following two questions:
\begin{enumerate}
\item[(Q1)]\label{Q1}
Under what conditions is $f_{\beta}$ identified?

\item[(Q2)]\label{Q2}
Does the random coefficients model impose restrictions?
\end{enumerate}

To answer these questions it is useful to introduce the notion of
the odd and even part of a function defined on the
sphere.
\begin{defn} We denote the odd part and the even part of a
function $f$ by
\begin{equation*}
f^-(b)=(f(b)-f(-b))/2
\end{equation*}
and
\begin{equation*}
f^+(b)=(f(b)+f(-b))/2,
\end{equation*}
respectively, for every $b$ in $\xSd$.
\end{defn}

Let us start with the question (Q1).   As noted in Section
\ref{sec:hemisp}, operating $\mathcal H$ reduces the even part of a
function to a constant 1 and therefore it is impossible to recover
$f_\beta^+$ from the knowledge of $r$, which is what observations
offer.  Our identification strategy is therefore as follows: (Step
1) Assume conditions that guarantee the identification of
$f_\beta^-$; then (Step 2) Show that $f_\beta$ is uniquely
determined from $f_\beta^-$ under a reasonable assumption.  We first
consider Step 1.  Define $H^+=H({\bf n}) = \{x\in\xSd:\ x'{\bf
n}\ge0\}$, where ${\bf n} = (1,,0,...,0)'$, that is, the northern
hemisphere of $\mathbb S^{d-1}$.  For later use, also define its
southern hemisphere $H^-=H(-\bf n)$.   Since the model we consider
has a constant as the first element of the covariate vector before
normalization, the same vector after normalization is necessarily an
element of $H^+$.  We make the following assumption, which also
appears in Ichimura and Thompson (1998), and show that it achieves
Step 1.
\begin{ass}\label{ass2i} The support of $X$ is  $H^+$.
\end{ass}
\noindent This assumption demands that $\tilde{X}$, the vector of
non-constant covariates in the original scale, is supported on the
whole space $\xR^{d-1}$. It rules out discrete or bounded
covariates; see Section \ref{sec:extensions} for a potential
approach to deal with regressors with limited support.  In what
follows we assume that the law of $X$ is absolutely continuous with
respect to $\sigma$ and denote its density by $f_X$. Step 1 of our
identification argument is to show that the knowledge of $r(x)$ on
$H^+$, which is available under Assumption \ref{ass2i}, identifies
$f_\beta^-$.   The problem at hand calls for solving $r = \mathcal H
f_\beta = \frac 1 2 + \mathcal H f_\beta^-$ for $f_\beta^-$, and the
inversion formula derived in (\ref{einv}) is potentially useful for
the purpose.  A direct application of the formula to $r$ is
inappropriate, however, since it requires integration of $r$ on the
whole sphere $\mathbb S^{d-1}$, but $r$ is defined only on $H^+$
even when $\tilde X$ has full support on $\mathbb R^{d-1}$.  An
appropriate extension of $r(x), x \in H^+$ to the entire $\mathbb
S^{d-1}$ is in order.  Using the random coefficients model
\eqref{emod} and Assumption \ref{ass1}, then noting that $f_{\beta}$
is a probability density function, conclude
\begin{equation}
\mathcal{H}(f_{\beta})(-x)=\int_{H(-x)}f_{\beta}(b)d\sigma(b)=
1-\mathcal H(f_\beta)(x) = 1-r(x)
\end{equation}
for $x$ in $H^+$.  This suggests an extension $R$ of $r$ to $\mathbb S^{d-1}$ as follows:
\begin{equation}\label{e2bb}
\forall x\in H^+, R(x)=r(x),\ {\rm and}\ \forall x\in H^-,
R(x)=1-r(-x)=1-R(-x). \end{equation} The function $R$ is
well-defined on the whole sphere under Assumption \ref{ass2i}.
Later we derive a formula for $f_\beta^-$ in terms of $R(x), x \in
\mathbb S^{d-1}$, which shows the identifiability of $f_\beta^-$
under Assumption \ref{ass2i}.

Note that
\begin{align}\label{echoice}
R(x)&=R^+(x)+R^-(x)
\\
&= \frac12 \left[R(x) + R(-x)\right] + R^-(x)\nonumber
\\
&=  \frac12 \left[R(x) +  (1 - R(x)) \right] + R^-(x) \quad \text{
by (\ref{e2bb})}\nonumber
\\
& = \frac12 +  R^-(x)
\nonumber
\end{align}
thus $R$ is completely determined by its odd part and therefore,
$$R(x)=\frac12+\mathcal{H}\left(f_{\beta}^-\right)(x),$$or
\begin{equation}\label{Randf}
R^- = \mathcal H f_\beta^-.
\end{equation}
We can invert this equation to obtain $f_\beta^-$.

Now we turn to Step 2 in our identification argument.  Obviously
$f_{\beta}^-$ does not uniquely determine $f_{\beta}$ without
further assumptions.  This is a fundamental identification problem
in our model.  We need to identify $f_\beta$ from the choice
probability function $r$, but we can choose an appropriate even
function $g$ so that $f_\beta + g$ is a legitimate density function
(see the proof of Proposition \ref{pQ2} for such a construction).
Then $r=\mathcal{H}\left(f_{\beta}+g\right)$, and the knowledge of
$r$ identifies $f_\beta$ only up to such a function $g$. Ichimura
and Thompson (1998, Theorem 1) give a set of conditions that imply
the identification of the model \eqref{emod}.  One of their
assumptions postulates that there exists $c$ on $\xSd$ such that
$\xP(c'\beta>0)=1$. This, in our terminology, means
that:
\begin{ass}\label{ass3} The support of $\beta$ is a subset of
some hemisphere.
\end{ass}
As noted by Ichimura and Thompson (1998), Assumption \ref{ass3} does
not seem too stringent in many economic applications.  It is often
reasonable to assume that an element of the random coefficients
vector, such as a price coefficient, has a known sign.  If the
$j$-th element of $\beta$ has a known sign (and positive), then
Assumption \ref{ass3} holds with $c$ being a unit vector with its
$j$-th element being 1.  This is a case in which the location of the
hemisphere in Assumption \ref{ass3} is known {\it a priori}, though
the knowledge about its location is not necessary for
identification. Assumption \ref{ass3} implies the following mapping
from $f_{\beta}^-$ to $f_{\beta}$ developed in
\eqref{backup}:
\begin{equation}\label{e15}
f_{\beta}(b)=2f_{\beta}^-(b)\I\left\{f_{\beta}^-(b)>0\right\}.
\end{equation}
This is useful because it shows that Assumption \ref{ass3}
guarantees identification if $f_{\beta}^-$ is identified.  Moreover,
it will be used in the next section to develop a key formula that
leads to a simple and practical estimator for $f_{\beta}$ that is
guaranteed to be non-negative.
\begin{rem}\label{r5}
Assumption \ref{ass3} is testable since it imposes restrictions on
$f_\beta^-$, which is identified under weak conditions.  For
example, for values of $b$ with $f_{\beta}^-(b) > 0$,
$f_{\beta}^-(-b) < 0$ must hold.  Or, it implies that $f_{\beta}^-$
integrates to $1/(2|\xSd|)$ on a hemisphere $H(x)$ for some $x$, and
$-1/(2|\xSd|)$ on the other $H(-x)$.
\end{rem}

The subsequent result, Proposition \ref{pQ2}, answers question (Q2),
and a proof is given in Supplemental Appendix.
\begin{notation} We use the notation $\xLtwo(\xSd)$ for the
space of square integrable complex valued functions equipped with
the hermitian product
$(f,g)_{\xLtwo(\xSd)}=\int_{\xSd}f(x)\overline{g}(x)d\sigma(x)$, and
more generally use $\xLn^p(\xSd)$ for $p\in[1,\infty]$ the Banach
space of $p$-integrable functions and $\|\cdot\|_p$ the
corresponding norm.   We also use the notation  $\xWn^s_p(\xSd)$
(and  $\xHn^s(\xSd)$ for $p=2$) to signify the corresponding Sobolev
spaces with norm $\|\cdot\|_{p,s}$ defined as
$$\|f\|_{p,s}=\|f\|_p+\left\|\left(-\Delta^S\right)^{s/2}f\right\|_p$$
where $\Delta^S$ denotes the Laplacian on the Sphere $\mathbb
S^{d-1}$: See Section \ref{sec:spheres} for further
discussions.
\end{notation}
\begin{prop}\label{pQ2}
A $[0,1]$-valued function $r$ is compatible with the random
coefficients model \eqref{emod} with $f_{\beta}$ in $\xLtwo(\xSd)$
and Assumption \ref{ass1} if and only if $r$ is homogeneous of
degree 0 and its extension $R$ according to \eqref{e2bb} belongs to
$\xHn^{d/2}(\xSd)$.
\end{prop}
The global smoothness assumption that $R$ belongs to
$\xHn^{d/2}(\xSd)$ imposes substantial restriction on the property
of observables, that is, the behavior of the choice probability
function $r$.  Note that the smoothness condition in this
proposition is stated in terms of $R$, and even if the choice
probability function $r$ is sufficiently smooth on the support of
$X$, which is $H^+$, it is not necessarily consistent with the
random coefficients binary choice model \eqref{emod} unless its
extension is smooth globally on $\mathbb S^{d-1}$.  In particular,
the Sobolev embedding of $\xHn^{s}(\xSd)$ into the space of
continuous functions for $s>(d-1)/2$ implies that if the extension
$R$ is in  $\xHn^{d/2}(\xSd)$, it has to be continuous on $\mathbb
S^{d-1}$.  This, in turn, means that the corresponding $r$ has to
satisfy certain matching conditions at a boundary point $x$ of $H^+$
(i.e. $x'\mathbf n = 0$) and its opposite point $-x$.

\section{Nonparametric Estimation of $f_\beta$}\label{sec:ParticularEstimatefBeta}

\subsection{Derivation of the closed form estimation formula}

\label{sec:asymptotics}  This section discusses how the closed form
estimation formula \eqref{mainestimator} is derived.
Suppose an odd
function $f^-$ defined on $\mathbb S^{d-1}$ satisfies an integral
equation $f^- = H g$ with $g$ square integrable with respect to the
spherical measure. In Section \ref{sec:hemisp} we show that the
solution to this equation is given by:
\begin{equation}\label{einv}
\mathcal{H}^{-1}(f^-)(y)= \sum_{p=0}^{\infty}\frac{1}{\lambda(2p+1,d)}
\int_{\mathbb S^{d-1}} q_{2p+1,d}(x,y)f^-(x)d\sigma(x)
\end{equation}
where expressions for $\lambda$ and $q$ are provided in Proposition
\ref{p3} and Theorem \ref{t1}, respectively. If an appropriate
estimator $\hat{R}^{-}$ of $R^-$ is available, an application of the
inversion formula \eqref{einv} to \eqref{Randf} suggests the
following estimator for $f_\beta^-$:
\begin{align}
\hat{f}_{\beta}^{-}&=\mathcal{H}^{-1}\left(\hat{R}^{-}\right)\label{eestb}\\
&=\sum_{p=0}^{\infty}\frac{1}{\lambda(2p+1,d)}\int_{\xSd}q_{2p+1,d}(\cdot,x)\hat{R}^{-}(x)d\sigma(x).\nonumber
\end{align}
Then use the mapping \eqref{e15} to define
\begin{equation}\label{eestfbeta}
\hat{f}_{\beta}(b)=2\hat{f}_{\beta}^{-}(b)\I\left\{
\hat{f}_{\beta}^{-}(b)>0\right\}
\end{equation}
as an estimator for $f_\beta$.

We use the following notation in the rest of the paper:
\begin{notation}
For two sequences of positive numbers $(a_n)_{n\in\xN}$ and
$(b_n)_{n\in\xN}$, we write $a_n\asymp b_n$ when there exists a positive $M$
such that $M^{-1}b_n\leq a_n\leq Mb_n$ for every
positive $n$.
\end{notation}
Proposition \ref{p4n} implies that  if $\hat f_\beta^- - f_\beta^- \in \xHn^{s}(\xSd)$ then
$\hat R^- - R^- \in \xHn^{\sigma}(\xSd)$,  $\sigma = s + \frac d 2$ and for $v\in[0,s]$,
\begin{equation}\label{eorder}
\|\hat{f}_{\beta}^{-}-f_{\beta}^-\|_{2,v}\asymp\|\hat{R}^{-}-R^-\|_{2,v+d/2}.
\end{equation}
  As discussed earlier, the estimation of $f_\beta$ is related to deconvolution in $\mathbb S^{d-1}$,
  and the degree of ill-posedness in our model is $d/2$, which is indeed the rate at which the absolute
  values of the eigenvalues of $\mathcal{H}$ (c.f. Proposition \ref{p3})
  $\lambda(n,d), n = 2p+1, p \in \mathbb N$ converges to zero as $p$ grows, as shown in \eqref{ein4}.
  Existing results for deconvolution problems (see, for example, Fan, 1991 and Kim and Koo, 2000)
  then suggest that we should be able to estimate $f_\beta$ at the rate  $N^{-\frac{s}{2s+2d-1}}$
  in the L$^2(\mathbb S^{d-1})$ provided that  $f_\beta \in \xHn^{s}(\xSd)$.  The relationship
  \eqref{eorder}, evaluated at $v = 0$, implies that this can be achieved if we can estimate $R^-$
  at the rate $N^{-\frac{\sigma - \frac d 2}{2\sigma+d-1}}$ in the $\|\cdot\|_{2,d/2}$ norm.  The
  latter is the usual nonparametric rate for estimation of densities on $d-1$ dimensional smooth
  submanifolds of $\mathbb R^d$ (see, for example, Hendriks, 1990).

The estimation formula given in \eqref{eestb} is natural and
reasonable, though it typically requires numerical evaluation of
integrals to implement it.  Moreover, in practice one needs to
evaluate the infinite sum in \eqref{eestb}, for example, by
truncating the series.  This results in a general estimator that can
be written in the following two equivalent forms
\begin{align}
\hat{f}_{\beta}^{-}&=\mathcal{H}^{-1}\left(P_{\tilde{T}_N}\hat{R}^{-}\right)\label{eestb2}\\
&=\sum_{p=0}^{T_N}\frac{1}{\lambda(2p+1,d)}\int_{\xSd}q_{2p+1,d}(\cdot,x)\hat{R}^{-}(x)d\sigma(x)\nonumber
\end{align}
for suitably chosen $\tilde{T}_N$ that goes to infinity with $N$ and
$P_{\tilde{T}_N}$ defined in \eqref{eq:projector}. The sequence
$\mathcal{H}^{-1}\circ P_{\tilde{T}_N}, N = 1,2,...$ can be
interpreted as regularized inverses of $\mathcal{H}$, with the
spectral cut-off method often used in statistical inverse problems.

We now discuss how to obtain $\hat R^-$ in the calculation of
\eqref{eestb2}.  The following choice is particularly convenient:
\begin{equation}\label{ert_2}
\hat{R}^{-}(x)=\frac{1}{N}\sum_{i=1}^N\frac{(2y_i-1)K_{2T_N}^{-}(x_i,x)}
{\max\left(\hat{f}_X(x_i),m_N\right)}
\end{equation}
where $m_N$ is a trimming factor going to 0 with the sample size,
$K^-(x_i,\cdot)$ denotes the odd part (of the second argument) of
the kernel function $K(x_i,\cdot)$ defined in \eqref{ekernel} and
$\hat f_X$ is a nonparametric density estimator for $f_X$.  See
Section \ref{sec:ChoiceProbability} of Supplemental Appendix for the
derivation of the above formula.  Various nonparametric estimators
for $f_X$ can be used in
\eqref{ert_2}, since estimation of densities
on compact manifolds have been studied by several authors, using
histogram (Ruymgaart (1989)), projection estimators (see, e.g.
Devroye and Gyorfi (1985) for the circle and Hendriks (1990) for
general compact Riemannian manifolds) or kernel estimators (see,
e.g. Devroye and Gyorfi (1985) for the case of the circle, and Hall
et al. (1987) and Klemel\"a (2000) for higher dimensional spheres).
Note also that Baldi et al. (2009) develops an adaptive density
estimator on the sphere using needlet thresholding. In the
simulation experiment we use
\begin{equation}\label{eestfX}
\hat{f}_X(x)=\max\left(\frac{1}{N}\sum_{i=1}^NK_{T_N'}(x_i,x),0\right)
\end{equation}
for a suitably chosen $T_N'$
that depends on the sample size and the
smoothness of $f_X$ and $K_{T_N'}$ is a kernel of the form
\eqref{ekernel} satisfying Assumption \ref{asskernel}.  Note that
its rate of convergence in sup-norm can be obtained in the same
manner as the proof of Theorem \ref{f_beta_rate}.  This estimator is
in the spirit of the projection estimators of Hendriks (1990), but
here we are able to derive a closed form using the condensed
harmonic expansions together with the Addition Formula.  Note also
that $K_{T_N}$ is a smoothed projection kernel (note the factor
$\chi$ in \eqref{ekernel}), which is used here in order to have good
approximation properties in the $\xLn^q(\xSd)$ norms with arbitrary
$q \in [1,\infty]$, in particular in the $\xLinfty(\xSd)$ norm.

Using \eqref{eestb2} and \eqref{ert_2} with $\tilde{T}_N=2T_N$,
define
$$\hat{f}_{\beta}^{-}=\mathcal{H}^{-1}\left(\hat{R}^{-}\right)
= \mathcal H^{-1} \left(
\frac{1}{N}\sum_{i=1}^N\frac{(2y_i-1)K_{2T_N}^{-}(x_i,\cdot)}
{\max\left(\hat{f}_X(x_i),m_N\right)} \right).$$ Computing $\hat
f^-_{\beta}$ is straightforward.  First, note that the estimator
\eqref{ert_2} for $R^-$ resides in a finite dimensional space
$\bigoplus_{p=0}^{T_N} H^{2p+1,d}$, therefore
$P_{2T_N}\hat{R}^{-}=\hat{R}^{-}$
holds. Consequently, unlike in
\eqref{eestb2} where a general estimator for $R^-$ is considered, we
do not need to apply any additional series truncation to $\hat R^-$
prior to the inversion of $\mathcal H$. Second, the estimator
requires no numerical integration.  To see this, note the formula
$$\mathcal{H}^{-1}\left(K_{2T_N}^{-}(x_i,\cdot)\right)(b)=\sum_{p=0}^{T_N - 1}
\frac{\chi(2p+1,2T_N)}{\lambda(2p+1,d)}q_{2p+1,d}(x_i,b),$$
which follows from
\begin{align*}
\int_{\mathbb S^{d-1}} q_{2p+1,d}(x,b)K_{2T_N}^{-}(x,x_i)d\sigma(x) &
= \int_{\mathbb S^{d-1}} q_{2p+1}(x,b)\sum_{p'=1}^{T_N - 1}\chi(2p'+1,2T_N)q_{2p'+1,d}(x,x_i)d\sigma(x)
\\
 & = \chi(2p+1,2T_N)q_{2p+1,d}(b,x_i).
\end{align*}
which, in turn, can be seen by the definition of $K_T$ in
\eqref{ekernel}, the fact that the integral operators with $q$
as
kernels are projections and \eqref{e7}.  Thus
$$
\hat{f}_{\beta}^{-}(b) =\frac{1}{N}\sum_{i=1}^N\frac{2y_i-1}{\max\left(\hat{f}_X(x_i), m_N\right)} \sum_{p=0}^{T_N - 1}
\frac{\chi(2p+1,2T_N)}{\lambda(2p+1,d)}q_{2p+1,d}(x_i,b).
$$
Using \eqref{eestfbeta} and the Addition formula (Theorem \ref{t1}), we arrive at an
estimator for $f_\beta$ with the following explicit form:
\begin{align}
&\hat f_\beta(b)  = 2 \hat f_\beta^-(b)\mathbb I\{\hat f_\beta^-(b) > 0\}\label{mainest},
\\
{\text{where }}& \hat f_\beta^-(b) = \frac 1 {|\mathbb S^{d-1}|}
\sum_{p=0}^{T_N - 1} \frac{\chi(2p+1,2T_N)h(2p+1,d)}{\lambda(2p+1,d)
C_{2p+1}^{\nu(d)}(1)} \left( \frac 1 N \sum_{i=1}^N \frac {(2y_i-1)
C_{2p+1}^{\nu(d)}(x_i'b)}{\max\left(\hat{f}_X(x_i), m_N\right)} \right). \nonumber
\end{align}
This is equivalent to the formula \eqref{mainestimator} previously
presented in Section \ref{sec:preliminaries}.  Likewise, using the
definition of the smoothing kernel \eqref{ekernel} and the Addition
Theorem in the above definition \eqref{eestfX} of $\hat f_X$, we
obtain the formula \eqref{fxest} as well.

\subsection{Rates of Convergence in $\xLn^q(\xSd)$-norms}\label{sec:rates}
Now we analyze the rate of our estimator $\hat f_\beta$.
The following assumption is weak and reasonable.
\begin{ass}\label{ass40}
$f_X\in\xLn^{\infty}$.
\end{ass}
The proofs of the following theorems and corollaries in the rest of
this section are given in Section \ref{sec:mainproofs} of
Supplemental Appendix.
\begin{thm}[Upper bounds in
$\xLn^q(\xSd)$]\label{f_beta_ub} Suppose Assumptions
\ref{asskernel}, \ref{ass2i} and \ref{ass40} hold, and choose $T_N$
that does not grow more than polynomially fast in $N$.  If
$f_{\beta}^-$ belongs to $\xWn_q^{s}(\xSd)$ with $q$ in $[1,\infty]$
and $s>0$, and
\begin{equation}\label{eCsimp}
\max_{i=1,\hdots,N}\left|f_X(x_i)-\hat{f}_X(x_i)\right|=O_p\left(m_N\right),
\end{equation}
then, for any $1\le r\le q$,
\begin{align}
\left\|\hat{f}_{\beta} - f_{\beta}\right\|_{q}=
O_p&\left(m_N^{-1}N^{-1/2}T_N^{(2d-1)/2}(\log N)^{(1/2-1/q)\I\{q\ge2\}}\right.\nonumber\\
&\ +T_N^{-s}+T_N^{d/2}m_N^{-2}\max_{i=1,\hdots,N}\left|f_X(x_i)-\hat{f}_X(x_i)\right|\nonumber\\
&\ \left.T_n^{d/2+(d-1)(1-1/r)}\sigma\left(f_X<m_N\right)^{1/q-1/r+1}\right)\label{eUB}.
\end{align}
When there exists $m>0$ such that $f_X\ge m\ \sigma\ a.e.$ on $H^+$, the following holds for
the estimator without the trimming factor (i.e. $m_N$ = 0) when the estimator $\hat{f}_X$ which
is consistent in sup norm:
\begin{align}
\left\|\hat{f}_{\beta} - f_{\beta}\right\|_{q}=
O_p&\left(N^{-1/2}T_N^{(2d-1)/2}(\log N)^{(1/2-1/q)\I\{q\ge2\}}\right.\nonumber\\
&\ \left.+T_N^{-s}+T_N^{d/2}\max_{i=1,\hdots,N}\left|f_X(x_i)-\hat{f}_X(x_i)\right|\right)\nonumber.
\end{align}
\end{thm}
The first term in \eqref{eUB} is the stochastic error, the second
term is the approximation bias, the third the plug-in error and the
fourth the trimming bias.  Note that Theorem \ref{f_beta_ub} imposes
the mild assumption \eqref{eCsimp}; otherwise, we need to
replace
$T_N^{d/2}m_N^{-2}\max_{i=1,\hdots,N}\left|f_X(x_i)-\hat{f}_X(x_i)\right|$
in \eqref{eUB} with
$T_N^{d/2}m_N^{-2}\max_{i=1,\hdots,N}\left|f_X(x_i)-\hat{f}_X(x_i)\right|
\left(1+(\log
N)^{(1/2-1/q)\I\{q\ge2\}}N^{-1/2}T_N^{(d-1)/2}\right)$. Since
$$\max_{i=1,\hdots,N}\left|f_X(x_i)-\hat{f}_X(x_i)\right|\le\left|f_X-\hat{f}_X\right|_{\infty},$$
this term can be made of order $O_P\left(\left(\frac{N}{\log
N}\right)^{-v/(2v+d-1)}\right)$ when $f_X\in \xWn_{\infty}^v$ with a
suitably chosen  parameter $T_N'$ if we take \eqref{eestfX} as an
estimator.  The proof of the latter statement is classical and can
be obtained simplifying the proof of Theorem \ref{f_beta_ub} and
Corollary \ref{f_beta_rate}. Equation \eqref{eUB} yields that, for
proper choices of $m_N$ going to zero and $T_N$ to infinity,
$\hat{f}_{\beta}$ is consistent given that $f_X$ has some smoothness
in the Sobolev scales.

Though the additional condition of $f_X$ being bounded away from 0
in the last statement of Theorem \ref{f_beta_ub} is convenient, it
is restrictive.  To see this, consider the $d=2$ case.  In polar
coordinates,
$f_X\left(\cos(\theta),\sin(\theta)\right)=f_{\tilde{X}}(\tan(\theta))(1+\tan^2(\theta)$,
thus, assuming $f_X\ge m$ on $H^+$, which does not require trimming,
yields $$\forall x\in\xR, f_{\tilde{X}}(x)\ge \frac{m}{1+x^2}.$$ It
implies that $\tilde{X}$ has tails larger than Cauchy tails and all
moments are infinite. The introduction of the trimming factor $m_N$
allows us to relax the assumption $f_X\ge m$, though it introduces
bias.  As is clear from \eqref{eUB}, the condition for the trimming
bias to go to zero with $N$ depends both on $T_N$ and $m_N$.  The
quantity $\sigma(f_X< m_N)$ should decay to zero with $N$
sufficiently fast. We can check, for example, that when $\tilde{X}$
is standard Gaussian then $\sigma(f_X< m_N)=O\left((-\log
m_N)^{-1/2}\right)$, when it is Laplace then $\sigma(f_X<
m_N)=O\left((-\log  m_N)^{-1}\right)$ and when $f_{\tilde{X}}$ is
proportional to $(1+x^2)^{-k}$ with $k>1$ we obtain that
$\sigma(f_X< m_N)=O\left(m_N^{1/(2(k-1))}\right)$. In all these
cases, it is possible to adjust adequately $T_N$ and $m_N$ and to
obtain rates of convergence.  The upper bound on the rates become
slower as the tail of $f_X$ becomes thinner.

Nonparametric estimation of the regression function with random degenerate design,
in the sense that the density of regressors can be low on its support, is
a difficult issue.  It has been studied for the pointwise risk in Hall et al. 1997, Ga\"iffas 2005,
Ga\"iffas 2009 and Guerre 1999.  Extension to inverse problems
setting is a widely open problem.   We tackle this problem for our specific inverse problem.
Future research includes the study of lower bounds from the minimax point
of view that account for the degeneracy of the design.

Let us now return to the general case of $d-1$ regressors. The
assumptions below allow us to obtain rates that differ slightly from
the rates that we would obtain in the ideal case where $f_X\ge m\
\sigma\ a.e.$ for positive $m$ on $H^+$.

\begin{ass}\label{ass4}
Suppose for $q$ in $[1,\infty]$, there exist positive $\tau$ and $r_X$
such that
\begin{enumerate}[\textup{(}i\textup{)}]
\item\label{condtrimb}
$\sigma(f_X< h)=O(h^{\tau})$ and $f_X\in\xLn^{\infty},$
\\
and either
\item\label{ass4i}
$$\max_{i=1,\hdots,N}\left|f_X(x_i)-\hat{f}_X(x_i)\right|
=O_p\left(\left(\frac{N}{(\log N)^{(1-2/q)\I\{q\ge2\}}}\right)
^{-r_X}\right)$$
\\
or,
\item\label{ass4ii} for some constant $C$,
$$\overline{\lim}_{N\rightarrow\infty}
\left(\frac{N}{(\log N)^{(1-2/q)\I\{q\ge2\}}}\right)
^{r_X}\max_{i=1,\hdots,N}\left|f_X(x_i)-\hat{f}_X(x_i)\right|
\le C\quad a.s.$$
\\
holds.
\end{enumerate}
\end{ass}
  As seen before,
Assumption \ref{ass4} \eqref{ass4i} or \eqref{ass4ii} are very mild. \eqref{condtrimb}
holds for a reasonable class of distributions for $f_X$. In the above example where $f_{\tilde{X}}$ is
proportional to $(1+x^2)^{-k}$ with $k>1$, we have the relation $\tau=\rho/(2(k-1))$.
This allows for a higher order moment to exist for a large $k$.

\begin{cor}\label{f_beta_rate}
Assume that $f_{\beta}^-$ belongs to $\xWn_q^{s}(\xSd)$ with $q$ in $[1,\infty]$
and $s>0$.  Let assumptions \ref{asskernel}, \ref{ass2i}, \ref{ass40} and \ref{ass4}
\eqref{condtrimb} and \eqref{ass4i} hold, and take
$$m_N\asymp\left(\frac{N}{(\log N)^{(1-2/q)\I\{q\ge2\}}}\right)^{-\rho},
\quad T_N\asymp \left(\frac{N}{(\log N)^{(1-2/q)\I\{q\ge2\}}}\right)^{\gamma(\rho)}$$
where $\rho$ yields a maximum
$\gamma$ of
$$\gamma(\rho)=\min\left(\frac{1-2\rho}{2s+2d-1},\frac{2\rho\tau}{2s+d+2(d-1)(1-1/q)},\frac{2r_X-4\rho}{2s+d},\frac{1}{d-1}\right).$$
We then have
\begin{equation}\label{eL2}
\left\|\hat{f}_{\beta} - f_{\beta}\right\|_{q}=
O_p\left(\left(\frac{N}{(\log N)^{(1-2/q)\I\{q\ge2\}}}\right)^{-\gamma s}\right).
\end{equation}
Moreover, if, instead of Assumption \ref{ass4} \eqref{ass4i}, Assumption \ref{ass4} \eqref{ass4ii}
holds with $q=\infty$, then there exists a constant $C$ such that
\begin{equation}\label{estronguniformconsistency}
\overline{\lim}_{N\rightarrow\infty}\left(\frac{N}{\log N}\right)^{\gamma s}
\left\|\hat{f}_{\beta} - f_{\beta}\right\|_{\infty}\le C\quad a.s.
\end{equation}
\end{cor}
The rate $\gamma s$ in Corollary \ref{f_beta_rate} accounts
for the dimension $d-1$,
the degree of smoothing $d/2$ of the operator
and features of the density of the covariates (i.e. its smoothness and tail behavior).

We now make stronger assumptions on $f_X$ and its estimate that yield,
up to a logarithmic term, the convergence rate $N^{-\frac{s}{2s+2d-1}}$.
We need to be able to trim the estimate of $f_X$ with a term which is
logarithmic in $N$: $m_N=(\log N)^{-\rho}$ for some positive
$\rho$.
\begin{ass}\label{ass4b}
Suppose for $q$ in $[1,\infty]$, and positive $r_{\sigma}$ and $r_X$,
\begin{enumerate}[\textup{(}i\textup{)}]
\item\label{ass4bi}
$\sigma(f_X< (\log N)^{-\rho})=O\left(\left(\frac{N}
{(\log N)^{2\rho+(1-2/q)\I\{q\ge2\}}}
\right)^{-r_{\sigma}}\right)$,
\\
and either
\item\label{ass4bii}
$$\max_{i=1,\hdots,N}\left|f_X(x_i)-\hat{f}_X(x_i)\right|
=O_p\left((\log N)^{-2\rho}\left(\frac{N}{(\log N)^{2\rho+(1-2/q)\I\{q\ge2\}}}\right)
^{-r_X}\right)$$
or,
\item\label{ass4biii} for some constant $C$,
$$\overline{\lim}_{N\rightarrow\infty}
(\log N)^{2\rho}\left(\frac{N}{(\log N)^{2\rho+(1-2/q)\I\{q\ge2\}}}\right)
^{r_X}\max_{i=1,\hdots,N}\left|f_X(x_i)-\hat{f}_X(x_i)\right|
\le C\quad a.s.$$
\end{enumerate}
\end{ass}

\begin{cor}\label{f_beta_rate2}
Assume that $f_{\beta}^-$ belongs to $\xWn_q^{s}(\xSd)$ with $q$ in $[1,\infty]$
and $s>0$.  Let assumptions \ref{asskernel}, \ref{ass2i}, \ref{ass40} and
\ref{ass4b} \eqref{ass4bi}-\eqref{ass4bii}, hold, and take
$$T_N\asymp \left(\frac{N}{(\log N)^{2\rho+(1-2/q)\I\{q\ge2\}}}\right)^{\gamma}$$
where
$$\gamma=\min\left(\frac{1}{2s+2d-1},\frac{2r_{\sigma}}{2s+d+2(d-1)(1-1/q)},\frac{2r_X}{2s+d}\right)$$
then we have
\begin{equation}\label{eL22}
\left\|\hat{f}_{\beta} - f_{\beta}\right\|_{q}=
O_p\left(\left(\frac{N}{(\log N)^{2\rho+(1-2/q)\I\{q\ge2\}}}\right)^{-\gamma s}\right).
\end{equation}
Moreover, if, instead of Assumption \ref{ass4b} \eqref{ass4bii}, Assumption \eqref{ass4biii}
holds with $q=\infty$, then there exists a constant $C$ such that
\begin{equation}\label{estronguniformconsistency2}
\overline{\lim}_{N\rightarrow\infty}\left(\frac{N}{(\log N)^{2\rho+1}}\right)^{\gamma s}
\left\|\hat{f}_{\beta} - f_{\beta}\right\|_{\infty}\le C\quad a.s.
\end{equation}
\end{cor}
When $f_X\in \xWn^{s+d/2+\epsilon}_{\infty}(\xSd)$ for any positive $\epsilon$
then $\frac{2r_X}{2s+d}>\frac{1}{2s+2d-1}$ and $\gamma$ in Corollary \ref{f_beta_rate2}
is simply $\min\left(\frac{1}{2s+2d-1},\frac{2r_{\sigma}}{2s+d+2(d-1)(1-1/q)}\right)$.
Recall that the smoothness $s+d/2$ is related
to the smoothness of $R$.
Indeed, we have seen in Section \ref{sec:deconvolution} that $R\in\xWn^{s+d/2}_2(\xSd)$
if and only if $f_{\beta}\in\xWn^{s}_2(\xSd)$.

Consider now the most restrictive case where $f_X\ge m\ \sigma\
a.e.$, then the estimator without the trimming factor (i.e. $m_N =
0$) satisfies the following: \begin{cor}\label{f_beta_rate3}
Assume
that $f_{\beta}^-$ belongs to $\xWn_q^{s}(\xSd)$ with $q$ in
$[1,\infty]$ and $s>0$.  Let assumptions \ref{asskernel},
\ref{ass2i} and \ref{ass40} hold, and suppose, for positive $r_X$,
\begin{equation}\label{ecor}
\max_{i=1,\hdots,N}\left|f_X(x_i)-\hat{f}_X(x_i)\right|=O_p\left(\left(\frac{N}{(\log N)^{(1-2/q)\I\{q\ge2\}}}\right)
^{-r_X}\right).
\end{equation}
Take
$$T_N\asymp \left(\frac{N}{(\log N)^{(1-2/q)\I\{q\ge2\}}}\right)^{\gamma}$$
where
$$\gamma=\min\left(\frac{1}{2s+2d-1},\frac{2r_X}{2s+d}\right)$$
then we have
\begin{equation}\label{eL23}
\left\|\hat{f}_{\beta} - f_{\beta}\right\|_{q}=
O_p\left(\left(\frac{N}{(\log N)^{(1-2/q)\I\{q\ge2\}}}\right)^{-\gamma s}\right).
\end{equation}
Moreover, if we replace \eqref{ecor} by for some positive $C$
\begin{equation}\label{ecor2}
\left(\frac{N}{(\log N)^{(1-2/q)\I\{q\ge2\}}}\right)
^{r_X}\max_{i=1,\hdots,N}\left|f_X(x_i)-\hat{f}_X(x_i)\right|\le C\quad a.s.
\end{equation}
then
\begin{equation}\label{estronguniformconsistency3}
\overline{\lim}_{N\rightarrow\infty}\left(\frac{N}{\log N}\right)^{\gamma s}
\left\|\hat{f}_{\beta} - f_{\beta}\right\|_{\infty}\le C\quad a.s.
\end{equation}
\end{cor}
When $f_X$ belongs to $\xWn_{\infty}^{s-d/2+\epsilon}$, for
arbitrary positive $\epsilon$, $\gamma=\frac{1}{2s+2d-1}$ in
Corollary \ref{f_beta_rate3}, and we recover the $\xLn^2$
convergence rate of $N^{\frac{s}{2s+2d-1}}$, the rate mentioned in
Section \ref{sec:asymptotics}.  It is in accordance with the
$\xLtwo$
rate in Healy and Kim (1996) who study deconvolution on
$\mathbb{S}^2$ for non-degenerate kernels. Kim and Koo (2000) prove
that the rate in Healy and Kim (1996) is optimal in the minimax
sense. Their statistical problem, however, involves neither a
plug-in method nor trimming.  Also, somewhat less importantly, it
does not cover the case when the convolution kernel is given by an
indicator function, which appears in our operator $\mathcal H$.
Hoderlein et al. (2010) study a linear model of the form $W =
X'\beta$ where $\beta$ is a $d$-vector of random coefficients.  They
obtain a nonparametric random coefficients density estimator that
has the $\xLn^2$-rate $N^{-\frac{s}{2s+2d-1}}$ when $f_X\ge m
\sigma\ a.e.$ for positive $m$\footnote{Note that the dimension of
their estimator is $d$, whereas that of ours is $d-1$.  On the other
hand, in their problem $W$ is observable, and it is obviously more
informative than our binary outcome $Y$, which causes difficulties
both in identification and estimation.} when $f_X$ is assumed to be
bounded from below and thus no trimming is required.  They also
consider trimming but the approach is slightly different and rates
of convergence are not given.  Unlike the previous results, we cover
L$^q$ loss for all $q \in [1,\infty]$.

\subsection{Pointwise Asymptotic Normality.}\label{sec:normality}
This section discusses the asymptotic normality property of our estimator.
\begin{thm}[Asymptotic normality]\label{t7}
Suppose $f_{\beta}^-$ belongs to $\xWn_{\infty}^{s}(\xSd)$ with
$s>0$, and Assumptions \ref{asskernel}, \ref{ass2i} and \ref{ass40} hold. If
$\hat{f}_X$, $f_X$, $m_N$ and $T_N$ satisfy
\begin{align}
&N^{1/2}T_N^{-(d-1)/2}m_N^{-2}\max_{i=1,\hdots,N}\left|f_X(x_i)-\hat{f}_X(x_i)\right|=o_p(1),\label{condPI}\\
&N^{-1/2}T_N^{(d-1)/2}m_N^{-(1+\epsilon)}=o(1)\quad{\rm for\ some\ }\epsilon>0,\label{condLyap}\\
&N^{1/2}T_N^{-\frac{2s + 2d - 1}2}=o(1),\label{edecay}\\
&N^{1/2}T_N^{(d-1)/2}\sigma\left(\left\{f_X< m_N\right\}\right)=o(1)\label{condapbias}
\end{align}
then
\begin{equation}\label{easnr}
N^{\frac 1 2}s^{-1}_N(b) \left(\hat{f}_{\beta}(b)-f_{\beta}(b)\right ) \stackrel d
\rightarrow N(0,1)\end{equation}
holds for $b$ such that $f_\beta(b) \neq 0$, where $s^2_N(b): =
\mathrm{var}(Z_{N}(b)), Z_{N}(b) = 2
\frac{(2Y-1)\mathcal{H}^{-1}\left(K_{2T_N}^{-}(X,\cdot)\right)(b)
}{\max(f_X(X),m_N)}$.
\end{thm}
The standard error $s_N(b)$ is the standard deviation of
\begin{equation}\label{eq:z_n}
Z_N(b) = \frac 2 {|\mathbb S^{d-1}|} \sum_{p=0}^{T_N - 1}
\frac{\chi(2p+1,2T_N)h(2p+1,d)}{\lambda(2p+1,d)C_{2p+1}^{\nu(d)}(1)}
\left(\frac {(2Y-1) C_{2p+1}^{\nu(d)}(X'b)}{\max(f_X(X),m_N)} \right)
\end{equation}
(see equation \eqref{mainest}), which can be estimated using an estimate $\hat f_X$ of
$f_X$.

The next theorem is concerned with the restrictive case where the density
of the covariates is bounded from below and hence the trimming factor $m_N$ is set at zero.
\begin{thm}[Asymptotic normality when the density of the covariates is bounded from below]\label{t7b}
Suppose $f_{\beta}^-$ belongs to $\xWn_{\infty}^{s}(\xSd)$ with
$s>0$, and Assumptions \ref{asskernel}, \ref{ass2i} and \ref{ass40} hold. If
$\hat{f}_X$, $f_X$ and $T_N$ satisfy
\begin{align}
&N^{1/2}T_N^{-(d-1)/2}\max_{i=1,\hdots,N}\left|f_X(x_i)-\hat{f}_X(x_i)\right|=o_p(1),\label{condPIb}\\
&N^{-1/2}T_N^{(d-1)/2}=o(1)\label{condLyapb}\\
&N^{1/2}T_N^{-\frac{2s + 2d - 1}2}=o(1),\label{edecayb}
\end{align}
then
\begin{equation}\label{easnr}
N^{\frac 1 2}s^{-1}_N(b) \left(\hat{f}_{\beta}(b)-f_{\beta}(b)\right ) \stackrel d
\rightarrow N(0,1)\end{equation}
holds for $b$ such that $f_\beta(b) \neq 0$, where
$s^2_N(b): =
\mathrm{var}(Z_{N}(b)), Z_{N}(b) = 2
\frac{(2Y-1)\mathcal{H}^{-1}\left(K_{2T_N}^{-}(X,\cdot)\right)(b)
}{f_X(X)}$.
\end{thm}
A formula for $Z_N$ for this case is obtained by replacing
$\max({f_X(X),m_N})$ with $f_X(X)$ in \eqref{eq:z_n}.

\section{Discussion}\label{sec:extensions}
\subsection{Estimation of Marginals}\label{sec:marginals}
In Section \ref{sec:deconvolution} we have provided an expression
for the estimator of the full joint density of $\beta$, from which
an estimator for a marginal density can be obtained.  Let $\sigma_k$
denote the surface measure and
$\underline{\sigma}_k=\sigma_k/|\mathbb{S}^k|$ the uniform
probability measure on $\mathbb{S}^k$. We write
$\beta=\left(\overline{\beta}',\overline{\overline{\beta}}'\right)'$
and wish to obtain the density of the marginal of $\overline{\overline{\beta}}$
which is a vector of dimension $k$.  Also define $\overline{P}$
and $\overline{\overline{P}}$ the projectors such that
$\overline{\beta}=\overline{P}\beta$ and
$\overline{\overline{\beta}}=\overline{\overline{P}}\beta$ and
denote by $\overline{P}_*\underline{\sigma}_{d-1}$ and
$\overline{\overline{P}}_*\underline{\sigma}_{d-1}$ the direct image
probability measures. One possibility is to define the marginal law
of $\overline{\overline{\beta}}$ as the measure
$\overline{\overline{P}}_*P_{\beta}$, where $dP_\beta = f_\beta
d\sigma$. This may not be convenient, however, since the uniform
distribution over $\mathbb S^{d-1}$ would have U-shaped marginals.
The U-shape becomes more pronounced as the dimension of $\beta$
increases. In order to obtain a flat density for the marginals of
the uniform joint distribution on the sphere it is enough to
consider densities with respect to the dominating measure
$\overline{\overline{P}}_*\underline{\sigma}_{d-1}$. Notice that
sampling $U$ uniformly on $\xSd$ is equivalent to sampling
$\overline{\overline{U}}$ according to
$\overline{\overline{P}}_*\underline{\sigma}_{d-1}$ and then given
$\overline{\overline{U}}$ forming
$\rho\left(\overline{\overline{U}}\right)V$ where $V$ is a draw from
the uniform distribution $\underline{\sigma}_{d-1-k}$ on
$\mathbb{S}^{d-1-k}$ and
$\rho\left(\overline{\overline{U}}\right)=\sqrt{1-
\left\|\overline{\overline{U}}\right\|^2}$. Indeed given
$\overline{\overline{U}}$,
$\overline{U}/\rho\left(\overline{\overline{U}}\right)$ is uniformly
distributed on $\mathbb{S}^{d-1-k}$. Thus, when $g$ is an element of
$\xLone(\xSd)$ we can write for $k$ in $\{1,\hdots,d-1\}$,
\begin{equation}\label{emarg0}\int_{\xSd}g(b)d\underline{\sigma}_{d-1}(b)
=\int_{\mathbb{B}^k}\left[\int_{\mathbb{S}^{d-1-k}}
g\left(\rho\left(\overline{\overline{b}}\right)u,\overline{\overline{b}}
\right)
d\underline{\sigma}_{d-1-k}(u)\right]
d\overline{\overline{P}}_*\underline{\sigma}_{d-1}\left(\overline{
\overline{b}}\right)
\end{equation}
where $\mathbb{B}^k$ is the $k$ dimensional ball of
radius 1.  Setting $g=|\xSd|f_{\beta}(b)\I\left\{\overline{
\overline{b}}\in A\right\}$
for $A$ Borel set of $\mathbb{B}^k$ shows that
the marginal density of
$\overline{\overline{\beta}}$ with respect to the dominating measure
$\overline{\overline{P}}_*\underline{\sigma}_{d-1}$ is given by
\begin{equation}\label{emarg}
{f}_{\overline{\overline{\beta}}}\left(\overline{\overline{b}}\right)
=|\xSd|\int_{\mathbb{S}^{d-1-k}} {f}_{\beta}\left(\rho\left(
\overline{\overline{b}}\right)u,
\overline{\overline{b}}\right)d \underline{\sigma}_{d - 1 - k}(u).
\end{equation}
One can use deterministic methods to compute the integral (e.g.,
Narcowich et al. (2006) for quadrature methods on the sphere) or for
example one may use a Monte-Carlo method, by forming
\begin{equation}\label{eestdensity}
\hat{f}_{\overline{\overline{\beta}}}^{M}\left(\overline{\overline{b}}\right)
=\frac{|\xSd|}{M}\sum_{j=1}^M\hat{f}_{\beta}
\left(\rho\left(\overline{\overline{b}}\right)u_j,\overline{\overline{b}}\right)
\end{equation}
where $u_j, j = 1,...,M$ are draws from independent uniform random variables on
$\mathbb{S}^{d-1-k}$.

\subsection{Treatment of Non-Random Coefficients}\label{sec:covariates}
It may be useful to develop an extension of the method described in
the previous sections to models that have non-random coefficients,
at least for two reasons.\footnote{Hoderlein et al. (2010) suggest a
method to deal with non-random coefficients in their treatment of
random coefficient linear regression models.}   First, the
convergence rate of our estimator of the joint density of $\beta$
slows down as the dimension $d$ of $\beta$ grows, which is a
manifestation of the curse of dimensionality.  Treating some
coefficients as fixed parameters alleviates this problem.  Second,
our identification assumption in Section \ref{sec:deconvolution} precludes
covariates with discrete or bounded support. This may not be
desirable as many random coefficient discrete choice models in
economics involve dummy variables as covariates.  As we shall see
shortly, identification is possible in a model where the
coefficients on covariates with limited support are non-random,
provided that at least one of the covariates with ``large support''
has a non-random coefficient as well.  More precisely, consider the
model:
\begin{equation}\label{nonrandom}
Y_i = \I\{\beta_{1i} + \beta_{2i}'X_{2i} + \alpha_1 Z_{1i} + \alpha_2'Z_{2i} \geq 0\}
\end{equation}
where $\beta_1 \in \R$ and $\beta_2 \in \R^{d_X-1}$ are random
coefficients, whereas the coefficients  $\alpha_1 \in \R$ and
$\alpha_2 \in \R^{d_Z-1}$ are nonrandom. The covariate vector
$(Z_1,Z_2')'$ is in $\R^{d_Z}$, though the $(d_Z - 1)$-subvector
$Z_2$ might have limited support: for example, it can be a vector of
dummies. The covariate vector $(X_{2}',Z_1)'$ is assumed to be,
among other things, continuously distributed.  Normalizing the
coefficients vector and the vector of covariates to be elements of
the unit sphere works well for the development of our procedure, as
we have seen in the previous sections.  The model (\ref{nonrandom}),
however, is presented ``in the original scale'' to avoid confusion.

Define $\beta^*_1(Z_2) := \beta_1 + \alpha_2'Z_2$.  We also use the notation
$$
\tau(Z_2) : =  \frac {(\beta_1^*(Z_2),\alpha_1,\beta_2)'} {\|(\beta_1^*(Z_2)
,\alpha_1,\beta_2')\|} \in \mathbb S^{d_X+1},
W := \frac {(1,Z_1,X_2')'} {\|(1,Z_1,X_2')'\|}  \in \mathbb S^{d_X+1}.
$$
Then (\ref{nonrandom}) is equivalent to:
\begin{align*}
Y
& = \I\{(\beta_1^*(Z_2),\alpha_1,\beta_2)(1,Z_1,X_2')' \geq 0 \}
\\
& = \I \left \{\tau(Z_2)'W \geq 0 \right \}.
\end{align*}
This has the same form as our original model if we condition on $Z_2
= z_2$.  We can then apply previous results for identification and
estimation under the following assumptions.  First, suppose
$(\beta_1,\beta_2')'$ and $W$ are independent, instead of Assumption
\ref{ass1}.  Second, we impose some conditions on $f_{W|Z_2 = z_2}$,
the conditional density of $W$ given $Z_2 = z_2$.  More
specifically, suppose there exists a set $\mathcal Z_2 \subset
\R^{d_Z -1}$, such that Assumption \ref{ass2i} holds if we replace
$f_X$ and $d$ with $f_{W|Z_2 = z_2}$ and $d_X+1$ for all $z_2 \in
\mathcal Z_2$.  If $Z_2$ is a vector of dummies, for example,
$\mathcal Z_2$ would be a discrete set.   By \eqref{e18} and
\eqref{einv} we obtain
\begin{equation}\label{edenscontrol}
f_{\tau(Z_2)|Z_2=z_2}^-(t)=\sum_{p=0}^{\infty}
\frac{1}{\lambda(2p+1,d_X+1)}
\xE\left[\left.\frac{(2Y-1)q_{2p+1,d_X+1}(W,t)}{{f}_{W|Z_2=z_2}(W)}
\right|Z_2=z_2\right]
\end{equation}
for all $z_2 \in \mathcal Z_2$, where the right hand side consists of observables.
This determines $f_{\tau(Z_2)|Z_2 = z_2}$.  That is, the conditional density
$$
f \left(\left . \frac {(\beta_1^*(Z_2),\alpha_1,\beta_2)} {\|(\beta_1^*(Z_2)
,\alpha_1,\beta_2)'\|} \right | Z_2 = z_2 \right )
$$
is identified for all $z_2 \in \mathcal Z_2$ (Here and henceforth we use
the notation $f(\cdot|\cdot)$ to denote conditional densities with appropriate
arguments when adding subscripts is too cumbersome). This obviously identifies
\begin{equation}\label{dist}
f \left(\left . \frac {(\beta_1^*(Z_2),\alpha_1,\beta_2)} {\|\beta_2\|}
\right | Z_2 = z_2 \right )
\end{equation}
for all $z_2 \in \mathcal Z_2$ as well.   If we are only interested in
the joint distribution of $\beta_2$ under a suitable normalization, we can
stop here.  The presence of the term $\alpha_1Z_1$ in (\ref{nonrandom}) is
unimportant so far.

Some more work is necessary, however, if one is interested in the joint
distribution of the coefficients on all the regressors.  Notice that the
distribution (\ref{dist}) gives
$$
f \left(\left . \frac {\beta_1^*(Z_2)} {\|\beta_2\|} \right | Z_2 = z_2
\right ) = f \left(\left . \frac {\beta_1 + \alpha_2'Z_2} {\|\beta_2\|}
\right | Z_2 = z_2 \right ),
$$
from which we can, for example, get
$$
\xE \left(\left . \frac {\beta_1^*(Z_2)} {\|\beta_2\|} \right | Z_2 = z_2
\right ) =  \xE \left( \frac {\beta_1} {\|\beta_2\|} \right ) +
\xE \left( \frac {1} {\|\beta_2\|} \right )\alpha_2'z_2 \quad
\text{for all } z_2 \in \mathcal Z_2.
$$
Define a constant
$$
c : =  \xE \left( \frac {1} {\|\beta_2\|} \right )
$$
then we can identify $c\alpha_2$ as far as $z_2 \in \mathcal Z_2$ has
enough variation and
$$\xE \left( \frac {\alpha_1} {\|\beta_2\|} \right ) = c\alpha_1$$
is identified as well.  Let
\begin{equation}\label{jointdist}
f\left( \frac{(\beta_2',\alpha_1,\alpha_2')'}{\|\beta_2\|} \right)
\end{equation}
denote the joint density of all the coefficient (except for $\beta_1$,
which corresponds to the conventional disturbance term in the original
model \eqref{nonrandom}, normalized by the length of $\beta_2$).  Then
\begin{align*}
f\left( \frac{(\beta_2',\alpha_1,\alpha_2')'}{\|\beta_2\|} \right)
=f\left (
\begin{bmatrix}
I_{d_X-1} & 0
\\
0 & 1
\\
\vdots & \frac {c\alpha_2}{c\alpha_1}
\end{bmatrix}
\begin{bmatrix}
\frac {\beta_2}{\|\beta_2\|}
\\
\frac{\alpha_1}{\|\beta_2\|}
\end{bmatrix}
\right ).
\end{align*}
In the expression on the right hand side, $f\left((\beta_2',\alpha_1)'
/\|\beta_2\|\right)$ is available from (\ref{dist}), and  $c\alpha_1$
and $c\alpha_2$ are identified already, therefore the desired joint
density (\ref{jointdist}) is identified.  Obviously (\ref{jointdist})
also determines the joint density of $(\beta_2',\alpha_1,\alpha_2')'$
under other suitable normalizations as well.

The density (\ref{edenscontrol}) is estimable: when $Z_2$ is discrete,
one can use the estimator of Section \ref{sec:ParticularEstimatefBeta} to each
subsample corresponding to each value of $Z_2$.  If $Z_2$ is continuous
we can estimate $f_{W|Z_2 = z_2}$ and the conditional expectation by
nonparametric smoothing.  An estimator for the density (\ref{dist})
can be then obtained numerically.

\subsection{Endogenous Regressors}\label{sec:endogeneity}
Assumption \ref{ass1} is violated if some of the regressors are
endogenous in the sense that the random coefficients and the
covariates are not independent.   This problem can be solved if an
appropriate vector of instruments is available.    To be more specific,
suppose we observe $(Y,X,Z)$ generated from the following model
\begin{equation}\label{maineq}
Y = \I\{\beta_1 +  \tilde{\beta}'X \geq 0\}
\end{equation}
with
\begin{equation}\label{reduceeq}
X=\Gamma Z+V
\end{equation}
where $V$ is a vector of reduced form residuals and $Z$ is independent of
$(\beta,V)$.  Note that Hoderlein et al. (2010) utilize a linear structure of
the form \eqref{reduceeq} in estimating a random coefficient linear model.
The equations (\ref{maineq}) and (\ref{reduceeq}) yield
$$
Y = \I\{\left(\beta_{1}+V'\tilde{\beta} \right)+Z'\Gamma'\tilde{\beta}\}.
$$
Suppose the distribution of  $\Gamma Z$ satisfy Assumption \ref{ass2i}.
It is then possible to estimate the density of
$\overline{\tau}=\tau/\|\tau\|$ where
$\tau=\left(\beta_{1}+V'\tilde{\beta},\tilde{\beta}\right)'$ by
replacing $\Gamma$ with a consistent estimator, which is easy to obtain
under the maintained assumptions.   This yields an estimator for the joint
density of $\tilde{\beta}/\|\tau\|$, the random coefficients on the
covariates under scale normalization.

\section{Numerical Examples}\label{sec:simulation}  The purpose of this section
is to illustrate the performance of our new estimator in finite samples using
simulated data.  We consider the model of the form \eqref{emod} with $d = 3$.
The covariates are specified to be $X = (1,X_1,X_2)$ where $(X_2,X_3)'
\sim N(\binom 0 0, 2\cdot I_2)$.  The coefficients vector $\beta = (\beta_1,\beta_2,1)'$
is set random except for the last element.  Fixing the last component constant fulfills
Assumption \ref{ass3} for identification.  Two specifications for the random elements
$(\beta_1,\beta_2)$ are considered.  In the first specification (Model 1) we let
$(\beta_1,\beta_2)' \sim  N(\binom 0 0, 0.3 \cdot I_2)$.  In the second (Model 2)
we consider a two point mixture of normals
$$
\binom{\beta_1}{\beta_2} \sim \lambda N\left(\binom \mu {-\mu}, \begin{bmatrix}
\sigma^2 & \rho \sigma^2
\\
\rho \sigma^2 & \sigma^2
\end{bmatrix}\right) + (1 - \lambda)N\left(\binom {-\mu} \mu, \begin{bmatrix}
\sigma^2 & \rho \sigma^2
\\
\rho \sigma^2 & \sigma^2
\end{bmatrix}\right),
$$ where
$\mu = 0.7, \sigma^2 = 0.3, \rho = 0.5$ and $\lambda = 0.5$.  Random samples of
size $500$ from each of the two specifications are generated, then the new
estimator \eqref{mainest} is computed.  It is implemented using the Riesz kernel
with $s = 2$ and $l = 3$ (see Proposition \ref{pkernel}).   The truncation
parameter $T_N$ is set at 3, and the trimming parameter $r$ is 2.  It also
requires a nonparametric estimator for $f_X$, and we use the projection estimator
\eqref{eestfX} based on the same Riesz kernel (i.e. $s = 2$, $l = 3$) and $T_N = 10$.
\begin{figure}[H]\label{fig1}
\includegraphics[height=.37\hsize]{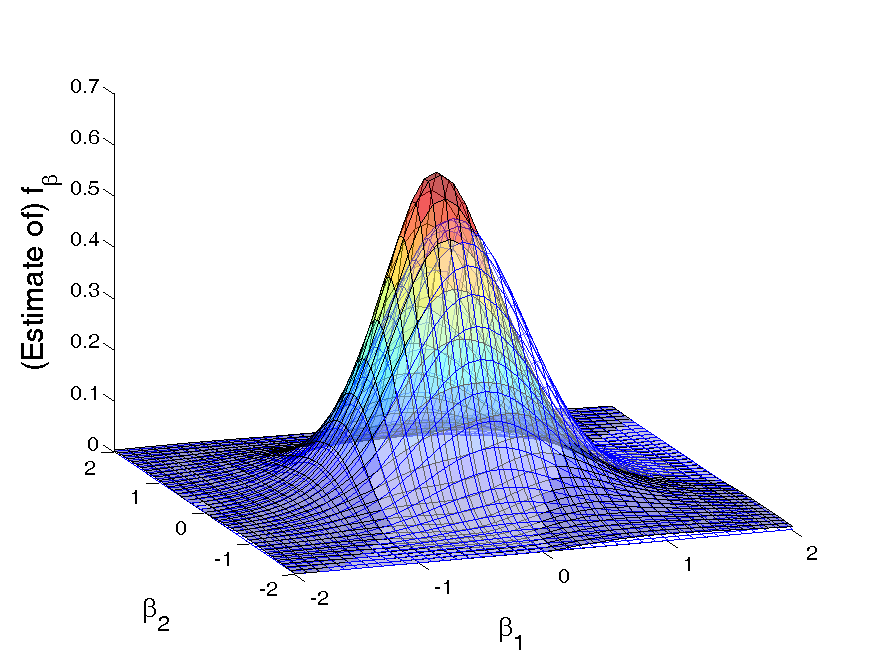}\includegraphics[height=.37\hsize]{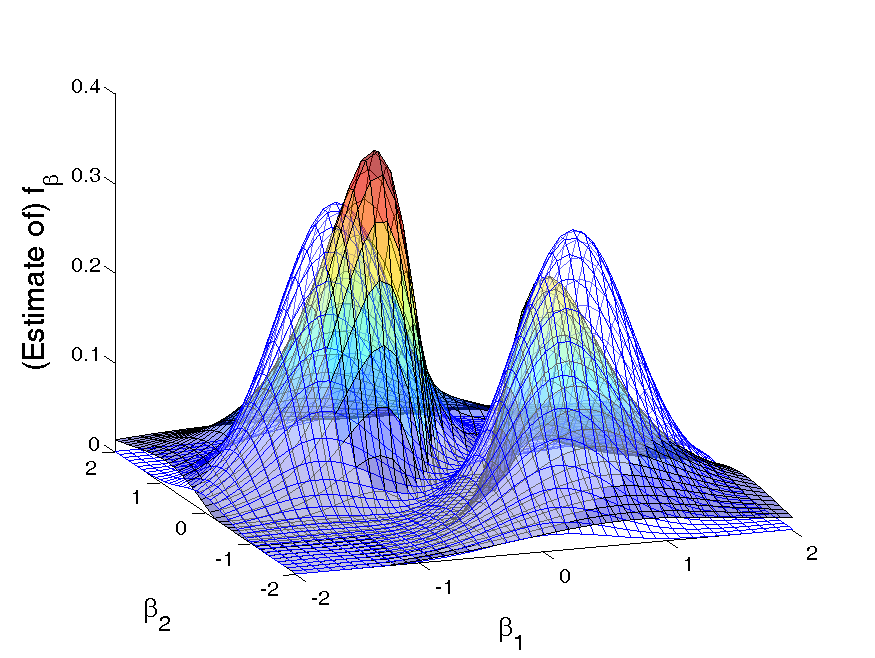}
\caption[11pt]{Nonparametric estimator of $f_{\beta}$ for Model 1 (left) and Model 2 (right)}
\end{figure}
Figure presents the surface plot of the true density (blue mesh) and our estimate
(multi-colored surface).  Our estimator \eqref{mainest} is defined on
$\mathbb S^2$ in this case, and we performed an appropriate transformation to plot
it as a density on $\mathbb R^2$.
In the case of model 1, with the reasonable sample size, the location of the peak
of the density, as well as its shape, are successfully recovered by our procedure.
For model 2, again, our procedure works well: the estimated
surface plot nicely captures the locations of the two peaks and their shapes of the
true density, thereby exhibiting the underlying mixture structure.  While further
experimentations are necessary, these results seem to indicate our estimator's good
performance in practical settings.

\section{Conclusion}\label{sec:conclusion}

In this paper we have considered nonparametric estimation of a random coefficients
binary choice model.
By exploiting (previously unnoticed) connections between the model and statistical
deconvolution problems and applying results of integral transformation on the sphere,
we have developed a new estimator that is practical and possesses desirable statistical
properties.  It requires neither numerical optimization nor numerical integration,
and as such its computational cost is trivial and local maxima and other difficulties
in optimization need not be of concern.  Its rate of convergence in the L$^q$ norm
for all $q \in [1,\infty]$ is derived.   Our numerical example suggests that the new
procedure works well in finite samples, consistent with its good theoretical properties.
It is of great theoretical and practical interest to obtain an adaptive procedure for
choosing the smoothing parameters of our estimator, though it is a task we defer to
subsequent investigations.\footnote{Gautier and Le Pennec (2011) consider a needlet-based
procedure and discuss its rate optimality in a minimax sense and adaptation.}
With
appropriate under-smoothing, the estimator is shown to be asymptotically normal,
providing a theoretical basis for nonparametric statistical inference for the random
coefficients distribution.

\clearpage
\bibliographystyle{econometrica}

{\footnotesize{

\noindent {\scshape{CREST (ENSAE), 3 avenue Pierre Larousse, 92245
Malakoff Cedex, France}}.

\noindent {\rm \it E-mail address}: {\ttfamily eric.gautier@ensae-paristech.fr}

\

\noindent {\scshape{Cowles Foundation for Research in Economics,
Yale University, New Haven, CT-06520.}}

\noindent {\rm \it E-mail address}: {\ttfamily yuichi.kitamura@yale.edu}}
}

\clearpage

\renewcommand{\thepage}{A-\arabic{page}}
\renewcommand{\theequation}{A.\arabic{equation}}
\renewcommand{\thelem}{A.\arabic{lem}}
\renewcommand{\thecor}{A.\arabic{cor}}
\renewcommand{\thedefn}{A.\arabic{defn}}
\renewcommand{\theprop}{A.\arabic{prop}}
\renewcommand{\therem}{A.\arabic{rem}}
\renewcommand{\thethm}{A.\arabic{thm}}
\renewcommand{\theass}{A.\arabic{ass}}
\renewcommand{\theass}{A.\arabic{ass}}
\renewcommand{\thesection}{A.\arabic{section}}
\setcounter{equation}{0}  
\setcounter{lem}{0}
\setcounter{cor}{0}
\setcounter{prop}{0}
\setcounter{rem}{0}
\setcounter{defn}{0}
\setcounter{lem}{0}
\setcounter{thm}{0}
\setcounter{ass}{0}
\setcounter{section}{1}
\setcounter{page}{1}
\setcounter{footnote}{0}
\section*{\bf SUPPLEMENTAL APPENDIX FOR ``NONPARAMETRIC ESTIMATION IN RANDOM COEFFICIENTS BINARY CHOICE MODELS''}  

\

\begin{center}

{\small ERIC GAUTIER AND YUICHI KITAMURA}

\end{center}

\

\subsection{A Toy Model}\label{sec:toy}
As noted in the main text, the key insight for our estimation procedure lies in the fact the estimation of $f_\beta$ in \eqref{e2} is mathematically equivalent to a statistical deconvolution problem.   To see this, it is useful to first consider the case with $d=2$.
We parameterize the vectors $b=\left(b_1,b_2\right)'$ and $x = \left(x_1,x_2\right)'$ on
$\mathbb{S}^1$ by their angles $\phi=\arccos\left(b_1\right)$ and  $\theta=\arccos\left(x_1\right)$ in
$[0,2\pi)$. As is often the case when Fourier series
techniques are used, we consider spaces of complex valued functions.
Let $\xLn^p(\mathbb{S}^1)$ denote the Banach space of Lebesgue
$p$-integrable functions and its norm by $\|\cdot\|_p$. In the case
of $\xLtwo(\mathbb{S}^1)$, the norm is derived from the hermitian
product $\int_0^{2\pi} f(\theta)\overline{g(\theta)}d\theta$.  Let $R_\theta$ and $f_\phi$ denote the extension $R$ of $r$ according to \eqref{e2bb} and $f_\beta$ after the reparameterization.  Our task is then to obtain $f_\phi$ from the knowledge of $R_\theta$.  Rewrite \eqref{e2} using these definitions, then divide both sides by $\pi$, to get:
\begin{equation}\label{e3b}
\frac {R_\theta} \pi(\theta) = \frac{\mathcal{H}(f_{\beta})}{\pi}(\theta)=\int_{0}^{2\pi}
\left(\frac{1}{\pi}\I\left\{|\theta-\phi|<\pi/2\right\}\right)
f_{\phi}(\phi)d\phi.
\end{equation}
If we further define $f_\theta := R_\theta/\pi$ and $f_\eta(\eta) :=
\frac 1 \pi \I\{|\eta| < \pi/2\}$, then using the standard notation
for convolution, \eqref{e3b} can be written as $f_\theta = f_\eta *
f_\phi$.  It is now obvious that the estimation of $f_\phi$ (thus
$f_{\beta}$) is linked to the following statistical deconvolution
problem: unobservable random variables $\phi$ and $\eta$ with
densities  $f_\phi$ and $f_\eta$ are related to an observable random
variable $\theta$ according to $\theta = \eta + \phi$, and one
wishes to recover $f_\phi$ from $f_\theta$, the density of $\theta$,
when $f_\eta$ is known (and it is Uniform$[-\pi/2,\pi/2]$ in this
case).\footnote{It is also useful to note that the inversion of
$\mathcal H$ is closely related to differentiation. Differentiating
the right hand-side of expression \eqref{e3b} with respect to
$\theta$ identifies $f_\phi(\theta+\pi/2)-f_\phi(\theta-\pi/2)$
where $f_\phi$ is defined on the line by periodicity. If $f_\phi$ is
supported on a semicircle, with an assumption that is elaborated
further in Section \ref{sec:deconvolution}, $f_\phi$ (which is
positive) is identified.   Thus if the model is identified the
inverse of $\mathcal H$ is a differential operator and as such
unbounded.}

The problem of deconvolution on the
unit circle can be conveniently solved using Fourier series. The set of functions $\left(\exp(-int)/\sqrt{2\pi}\right)_{n\in\xZ}$
is the orthonormal basis of $\xLtwo(\xS1)$ used to define Fourier series.
This system is also complete in $\xLone(\xS1)$.  Reparameterize a function $f\in\xLone(\xS1)$ it using angles as above, and denote it by $f_t$.  Denoting the
Fourier coefficients of $f\in\xLone(\xS1)$ by $c_n(f_t)=\int_0^{2\pi}f_t(t)\exp(-int)dt/(2\pi)$,
\begin{equation}\label{efourier}
f_t(\theta)=\sum_{n\in\xZ}c_n(f_t)\exp(int)
\end{equation}
holds in the $\xLone(\xS1)$ sense. Recall also that for $f$ and $g$ in
$\xLone(\xS1)$, after the same reparameterization,
\begin{equation}\label{econvolution}
c_n(f_t*g_t)=2\pi c_n(f_t)c_n(g_t).
\end{equation}
Using equation \eqref{econvolution} we obtain the following
proposition.
\begin{prop}\label{p2}
$c_0(R_\theta)=\pi c_0\left(f_\phi\right)$ and for
$n\in\xZ\setminus\{0\}$, $c_n(R_\theta)=
c_n\left(f_\phi\right)2\sin\left(n\pi/2\right)/n.$
\end{prop}

As in classical deconvolution problems on the real line, our aim is to obtain $f_t$ (thus $f_\beta$) using equation \eqref{efourier} and
Proposition \ref{p2}.  Proposition \ref{p2} shows that $c_{2p}(R_\theta)=0$
holds for all non-zero $p$'s, regardless of the values of
$c_{2p}(f_\phi), p \in \mathbb Z \backslash \{0\}$.  Thus from $r(x) = R_\theta(\theta)$ one can only recover the Fourier coefficients $c_n(f_\phi)$ for $n = 0$ (which is easily seen to be $1/2\pi$,
by integrating both sides of (\ref{p2}) and noting that $f_\beta$ is a probability density function) and $n = 2p + 1, p \in \mathbb Z$.
The same phenomenon occurs in higher dimensions, as explained in
Section \ref{sec:spheres}.

\begin{rem}\label{rsf}
The vector spaces
$H^{2p+1,2}=\spn\left\{\exp(i(2p+1)t)/\sqrt{2\pi},\exp(-i(2p+1)t)/\sqrt{2\pi}\right\}, p \in \mathbb N$
are eigenspaces of the compact self-adjoint operator
$\mathcal{H}$ on L$^2(\mathbb S^1)$.  These eigenspaces are associated with the
eigenvalues $\frac{2(-1)^p}{2p+1}$.  Also,
$\bigoplus_{p\in \mathbb N \backslash\{0\}}H^{2p,2}$ is the null space $\xker\
\mathcal{H}$.
\end{rem}

\subsection{The Gegenbauer polynomials}\label{sec:premresults}
We summarize some results on the Gegenbauer polynomials, which are used in various parts of the paper.
These can be found in Erd\'elyi et al. (1953) and Groemer
(1996).
When $\nu=0$
and $d=2$, it is related to the Chebychev polynomials of the first
kind, as
$$\forall n\in\xN\setminus\{0\},\ C_{n}^0(t)=\frac{2}{n}T_{n}(t)$$
and $$C_0^0(t)=T_0(t)=1$$ hold for $$T_{n}(t)=\cos\left(n\arccos(t)\right), n \in \mathbb N.$$
When $\nu=1$ and $d=4$,
$C_{n}^{1}(t)$ coincides with the Chebychev polynomial of the second
kind $U_n(t)$, which is given by
\begin{equation*}
U_n(t)=\frac{\sin[(n+1)\arccos(t)]}{\sin[\arccos(t)]}, n \in \mathbb N.
\end{equation*}
The Gegenbauer polynomials are orthogonal with respect to the weight function
$(1-t^2)^{\nu-1/2}dt$ on $[-1,1]$.  Note that $C_0^{\nu}(t)=1$ and
$C_1^{\nu}(t)=2\nu t$ for $\nu\ne0$ while $C_1^{0}(t)=2t$.  Moreover, the following
recursion relation holds
\begin{equation}\label{e11}
(n+2)C_{n+2}^{\nu}(t)=2(\nu+n+1)tC_{n+1}^{\nu}(t)-(2\nu+n)C_n^{\nu}(t).
\end{equation}
Implementation of our estimator requires evaluation of the Gegenbauer polynomials
for a series of successive values of $n$.  The recursion relation \eqref{e11} is
therefore a powerful tool.  The Gegenbauer polynomials are related to each other through
differentiation, that is, they satisfy
\begin{equation}\label{estability1}
\frac{\xdif}{\xdif t}C_{n}^{\nu}(t)=2\nu C_{n-1}^{\nu+1}(t)
\end{equation}
for $\nu>0$ and
\begin{equation}\label{estability2}
\frac{\xdif}{\xdif t}C_{n}^{0}(t)=2 C_{n-1}^{1}(t).
\end{equation}
For $\nu\ne0$ the Rodrigues formula states that
\begin{equation}\label{e9}
C_n^{\nu}(t)=(-2)^{-n}(1-t^2)^{-\nu+1/2}\frac{(2\nu)_n}{(\nu+1/2)_nn!}\frac{\xdif^n}{\xdif
t^n}(1-t^2)^{n+\nu-1/2}.
\end{equation}
The following results are also used in the paper:
\begin{equation}\label{einfty}
\sup_{t\in[-1,1]}\left|\frac{C_n^{\nu}(t)}{C_n^{\nu}(1)}\right|\leq1,
\end{equation}
\begin{equation}\label{eval11}
\forall\ \nu>0,\ \forall n\in\xN,\
C_n^{\nu}(1)=\left(\begin{array}{c}n+2\nu-1\\n\end{array}\right)
\end{equation}
\begin{equation}\label{eval12}
C_0^0(1)=1\ {\rm and}\ \forall n\in\xN\setminus\{0\},\
C_n^{0}(1)=\frac2n,
\end{equation}
\begin{equation}\label{epar}
C_n^{\nu}(-t)=(-1)^nC_n^{\nu}(t).
\end{equation}
These orthogonal polynomials are normalized such that
\begin{equation}\label{emass}
\|C_n^{\nu(d)}\left(x'\cdot\right)\|_{2}^2=|\mathbb{S}^{d-2}|\int_{-1}^1(C_n^{\nu(d)}(t))^2(1-t^2)^{(d-3)/2}dt=\frac{|\xSd|
(C_n^{\nu(d)}(1))^2}{h(n,d)}.
\end{equation}

\subsection{Tools for Higher Dimensional Spheres}\label{sec:spheres}
Let us introduce some concepts used for the treatment of the general case $d\ge2$.
We consider functions defined on the
sphere $\xSd$, which is a $d-1$ dimensional smooth submanifold in $\xR^d$. The canonical measure on $\xSd$ (or the spherical measure) is
denoted by $\sigma$.  It is a uniform measure on $\xSd$ satisfying $\int_{\xSd}d\sigma=|\xSd|$, where $|\xSd|$ signifies the surface area of the unit sphere.

Recall that the basis functions $\exp(\pm int)/\sqrt{2\pi}$ are
eigenfunctions of $-\frac{\xdif}{\xdif t^2}$ associated with eigenvalue $n^2$.  In a similar way, the Laplacian on the sphere $\mathbb S^{d-1}, d \geq 2$, denoted by $\Delta^S$, can be used to obtain an orthonormal basis for higher dimensional spheres.  It can be defined by the formula
\begin{equation}\label{slapl}
\Delta^S f = (\Delta f\:\check{})\hat{}
\end{equation}
where $\Delta$ is the Laplacian in $\R^d$,
$f\check{}$ the radial extension of $f$,
that is $f\check{}(x) = f(x/\|x\|)$, and $f\hat{}$
the restriction of $f$ to $\xSd$.  Likewise the gradient on the sphere is given by:
\begin{equation}\label{sgrad}
\nabla^S f = (\nabla f\:\check{})\hat{}
\end{equation}
where $\nabla$ is the gradient in $\R^d$.

\begin{defn}\label{d1}A surface harmonic of degree
$n$ is the restriction of a homogeneous harmonic
polynomial (a homogeneous polynomial $p$ whose Laplacian $\Delta p$ is zero) of degree $n$ in $\xR^d$ to $\mathbb S^{d-1}$.
\end{defn}
The reader is referred to M\"uller (1966) and Groemer (1996) for clear and detailed
expositions on these concepts and important results concerning
spherical harmonics used in this paper. Erd\'elyi et al. (1953, vol.
2, chapter 9) provide detailed accounts focusing on special
functions.  Here are some useful results:
\begin{lem}\label{l2} The following properties hold:
\begin{enumerate}[\textup{(}i\textup{)}]

\item\label{l2i}
$-\Delta^S$ is a positive self-adjoint unbounded operator on
$\xLtwo(\xSd)$, thus it
has orthogonal eigenspaces and a basis of eigenfunctions;

\item\label{l2ii}
Surface harmonics of degree $n$ are eigenfunctions of $-\Delta^S$
for the eigenvalue $\zeta_{n,d} := n(n+d-2)$;

\item\label{l2iv} The dimension of the vector space $H^{n,d}$ of
surface harmonics of degree $n$ is
\begin{equation}\label{ehnd}
h(n,d):=\frac{(2n+d-2)(n+d-2)!}{n!(d-2)!(n+d-2)};
\end{equation}

\item\label{l2iii}
A system formed of orthonormal bases $(Y_{n,l})_{l=1}^{h(n,d)}$ of $H^{n,d}$
for each degree $n=0,\hdots,\infty$ is complete in $\xLone(\xSd)$, that is,
for every $f\in\xLone(\xSd)$ the following equality holds in the $\xLone(\xSd)$ sense:
$$f=\sum_{n=0}^{\infty}\sum_{l=1}^{h(n,d)}\left(f,Y_{n,l}\right)_{\xLtwo(\xSd)}Y_{n,l}.$$

\end{enumerate}
\end{lem}
Thus $h(n,d)$ is the multiplicity of the eigenvalue $\zeta_{n,d}$,
and $H^{n,d}$ is the corresponding eigenspace.   Lemma \ref{l2}
(\ref{l2i}), (\ref{l2ii}) and (\ref{l2iii}) give the
decomposition
$$\xLtwo(\xSd)=\bigoplus_{n\in\xN}H^{n,d}.$$ The space
of surface harmonics of degree 0 is the one dimensional space
spanned by $1$. A series expansion on an orthonormal basis of
surface harmonics is called a Fourier series when $d=2$, a Laplace
series when $d=3$ and in the general case a Fourier-Laplace series.

\indent Orthonormal bases of surface harmonics usually involve
parametrization by angles, such as the spherical coordinates when
$d=3$ as used by Healy and Kim (1996) or hyperspherical coordinates
for $d>3$. Instead, here we work with the decomposition of a
function on the spaces $H^{n,d}$ as presented in the next definition
so that we avoid specific expressions of basis
functions.
\begin{defn} The condensed harmonic expansion of a
function $f$ in $\xLone(\xSd)$ is the series
$\sum_{n=0}^{\infty}Q_{n,d}f$, where $Q_{n,d}$ is the projector from
L$^2(\mathbb S^{d-1})$ to $H^{n,d}$.
\end{defn}
This leads to a simple method both in terms of theoretical developments and
practical implementations. The projector $Q_{n,d}$ can be expressed as an
integral operator
with kernel
\begin{equation}\label{e7}
q_{n,d}(x,y)=\sum_{l=1}^{h(n,d)}\overline{Y_{n,l}(x)}Y_{n,l}(y),
\end{equation}
where $(Y_{n,l})_{l=1}^{h(n,d)}$ is any orthonormal basis of $H^{n,d}$.
The kernel has a simple expression
given by the addition formula:
\begin{thm}[Addition Formula]\label{t1} For every $x$ and $y$ $\in \mathbb S^{d-1}$,
we have
\begin{equation}\label{e8}
q_{n,d}(x,y)=  {}^\flat q_{n,d}(x'y), \quad  {}^\flat q_{n,d}(t) :=
\frac{h(n,d)C_n^{\nu(d)}(t)}{|\xSd|C_n^{\nu(d)}(1)}
\end{equation} where $C_n^{\nu}$ are Gegenbauer a and  $\nu(d)=(d-2)/2$.
\end{thm}
The Sobolev spaces are defined in the Fourier-Laplace domain through the fractional
Laplacian defined on a certain subset of $\xLn^p(\xSd)$ as
\begin{equation}\label{laplacian}
\left(-\Delta^S\right)^{s/2}f:=\sum_{n=0}^{\infty}\zeta_{n,d}^{s/2}Q_{n,d}f.
\end{equation}
For the case where $p = 2$, in stead of the definition of the norm
$\|\cdot\|_{p,s}$ given in Section \ref{sec:deconvolution} it is
also possible to use an equivalent norm, the square of which is
equal to
$$\sum_{n=0}^{\infty}\left(1+\zeta_{n,d}\right)^{s}\|Q_{n,d}f\|_2^2.$$
The following integration by parts holds for functions $f$ in $\xHn^1(\xSd)$
\begin{equation}\label{eipp}
-\int_{\xSd}f(x)\Delta^S f(x)d\sigma(x)=\int_{\xSd}\nabla^S_xf'\nabla^S_xfd\sigma(x)
\end{equation}
and as a consequence for the second definition of the norm of $\xHn^1(\xSd)$ we have
$$\|f\|_{2,1}^2=\|f\|_{2}^2+\|\nabla^Sf\|_{2}^2.$$

In Section \ref{sec:toy} we observed the close relationship between
the random coefficients binary choice model and convolution for
$d=2$.  This connection remains valid in higher dimensions.
Suppose a function  $f(x,y)$ defined on $\mathbb S^{d-1} \otimes
\mathbb S^{d-1}$ depends on $x$ and $y$ only through the spherical
distance $d(x,y)=\arccos(x'y)$ (that is, $f$ is a zonal
function).
Consider the following integral:
$$
h(x) =  \int_{\xSd}f(x,y)g(y)d\sigma(y): = f \ast g (x),
$$
then the function $h$ is a convolution on the sphere.   We now see
that the choice probability function $r(x) = \mathcal H(f_\beta)(x)
= \int_{\xSd}\mathbb I\{x'b \geq 0\}f_\beta(b)d\sigma(b)$ is a
special case of $h$ and therefore can also be regarded as
convolution.  Obtaining $f_\beta$ from $r$ (or, inverting
$\mathcal{H}$) is  therefore a deconvolution problem.

In what follows we often write $f(x,\star)$ when a function $f$ on
$\mathbb S^{d-1} \otimes \mathbb S^{d-1}$ is regarded as a function
of $\star$.  Also, the notation $\|f(x,\star)\|_p$ is used for the
L$^p$ norm of $f(x,\star)$, that is, $\|f(x,\star)\|_p =
\int_{\mathbb S^{d-1}}|f(x,y)|^pd\sigma(y)$.  Note that if $f$ is a
zonal function as in the above definition of spherical convolution,
its $\mathrm L^p$ norm $\|f(x,\star)\|_p$ does not depend on $x$.
The following Young inequalities for convolution on the sphere (see,
for example, Kamzolov, 1983) are useful:

\begin{prop}[Young inequalities]\label{pYoung}
Suppose $f(x,\star)$ and $g$ belong to $\xLn^r(\xSd)$ and
$\xLn^p(\xSd)$, respectively.  Then $h(x) =  f \ast g(x)$ is
well-defined in
$\xLn^q(\xSd)$ and
$$\|h\|_q\le\|f\|_r\|g\|_p,$$
where $1 \leq p,q,r \leq \infty$ and $\frac 1 q = \frac 1 p + \frac 1 r -1$.
\end{prop}

Let $P_T$ denote the projection operator onto $\bigoplus_{n=0}^TH^{n,d}$, i.e.
\begin{equation}\label{eq:projector}
P_Tf(x)= \sum_{n = 0}^T Q_{n,d}f(x) =  \int_{\xSd}D_T(x,y)f(y)d\sigma(y)
\end{equation}
where $$D_T(x,y)=\sum_{n=0}^Tq_{n,d}(x,y).$$ The kernel $D_T$
extends the classical Dirichlet kernel on the circle to the sphere
$\xSd$. The sum over $T$ in the definition of $D_T$ also has the
simple closed form in terms of derivatives of Gegenbauer
polynomials; see Equation (52) in M\"uller (1966).  The linear form
$f\rightarrow  \int_{\xSd}D_T(x,y)f(y)d\sigma(y)$ converges to
$\int_{\xSd}f(y)d\delta_x(y)  = f(x)$  as $T$ goes to infinity,
where $\delta_{x}$ denotes the Dirac measure. The Dirichlet kernel
yields the best approximation $P_T f$ of $f$ in $\xLtwo(\xSd)$ by
polynomials that belong to $\bigoplus_{n=0}^TH^{n,d}$, but is known
to have flaws.
For example, $D_T$ does not satisfy $$\forall
f\in\xLone(\xSd),
\lim_{T\rightarrow\infty}\left\|D_T*f-f\right\|_{\xLone(\xSd)}=0,$$
that is, the sequence $D_T, T = 0,1,...$ is not an approximate
identity (see, e.g., Devroye and Gyorfi 1985) in $\xLone(\xSd)$.
Indeed, the $\xLone(\xSd)$ norm of the kernel is not uniformly
bounded; more precisely, we have
\begin{equation}\label{el1dirichlet}
\left\|D_T(\cdot,x)\right\|_{1}\asymp T^{(d-2)/2}
\end{equation}
when $d\ge3$ and
\begin{equation}\label{el1dirichletb}
\left\|D_T(\cdot,x)\right\|_{1}\asymp \log T
\end{equation}
when $d=2$ (as noted above, these norms do not depend on the value
of $x \in \xSd$).   These bounds can be found in Gronwall (1914) for
$d=3$ and Ragozin (1972) and Colzani and Traveglini (1991) for
higher dimensions.  Also, $D_T$ does not have good approximation
properties in $\xLinfty(\xSd)$; in particular, we do not
have
$$\forall f\in\xLinfty(\xSd),
\lim_{T\rightarrow\infty}\left\|D_T*f-f\right\|_{\xLinfty(\xSd)}=0.$$
Near the points of discontinuity of $f$, $D_T * f$ has oscillations
which do not decay to zero as $T$ grows to infinity, known as the
Gibbs oscillations.  This phenomenon deteriorates as the dimension
increases.   These problems can be addressed by using kernels that
involves extra smoothing instead of the Dirichlet kernel $D_T$.  To
this end, define a general class of kernel
\begin{equation}\label{ekernel}K_{T}(x,y)=\sum_{n=0}^T\chi(n,T)q_{n,d}(x,y)\end{equation}
for some sequence $\chi(n,T)$.  These are called smoothed projection
kernels.  Typically the function $\chi$ is chosen so that it puts
more weight on lower frequencies.  In particular we impose the
following conditions:
\begin{ass}\label{asskernel}
\begin{enumerate}[\textup{(}i\textup{)}]
\item\label{k1}
$\left\|K_{T}(x,\star)\right\|_{1}$ is uniformly bounded in $T$.
\item\label{k2} There exists constants $C$ and $\alpha$ such that for all
$x,y,z\in\xSd$,
$$\left| K_{T}(z,x)- K_{T}(z,y)\right|\le C\|x-y\|T^{\alpha},$$
where $\|\cdot\|$ denotes the Euclidean norm.
\item \label{k3} For $p \in [1,\infty]$ and $s > 0$, there exists a
constant $C$ such that for every $f$ in $\xWn_p^s(\xSd)$,
$$\left\|f(\cdot)- \int_{\xSd}K_T(\cdot,y)f(y)d\sigma(y)\right\|_{p}\le CT^{-s}
\left\|f\right\|_{p,s}.$$
\item\label{k4} $\chi(\cdot,T)$ takes values in $[0,1]$ and is such that there exists $c > 0$ such that
for all $0\le n\le\lfloor T/2\rfloor$, $\chi\left(n,T\right)\ge c$.
\end{enumerate}
\end{ass}

The smoothed projection kernel $K_{T}(x,y)$ depends on $x$ and $y$
only through $d(x,y)$, thus the value of the norm
$\left\|K_{T}(x,\star)\right\|_{1}$ in Assumption \eqref{k1} does
not depend on $x \in \xSd$.  Assumption \eqref{k1} could be relaxed,
but imposing this on $K_T$ allows us to make relatively weak
assumptions on the smoothness of the density of the covariates later
in this paper.  Assumption \eqref{k2} is used to establish
the
$\mathrm L^{\infty}$-rates of convergence of our estimators.
Assumption \eqref{k3} provides bounds for approximation errors.
Under this condition, $K_T \ast f$ approximates $f \in
\xLn^p(\xSd)$ with an error of the same order as that of the best
$n$-th degree spherical harmonic approximation of a function $f \in
\xLn^p(\xSd)$ in $\xWn_p^s(\xSd)$ (see e.g. Kamzolov 1983 and
Ditzian 1998).  This is useful in our treatment of the bias terms in
our estimators.  As concrete examples, the following two choices for
the weight function $\chi$ in \eqref{ekernel} satisfy Assumption
\ref{asskernel}, as shown in the appendix.  The first and the second
choices of $\chi$ correspond to the {\it{Riesz kernel}} and the
{\it{delayed means kernel}}, respectively.

\begin{prop}\label{pkernel}  In the definition of the smoothed
kernel \eqref{ekernel}, let
$$\chi(n,T)=\left(1-\left(\frac{\zeta_{n,d}}{\zeta_{T,d}+1}\right)^{s/2}\right)^l,$$
where $l$ is an integer satisfying $l>(d-2)/2$, or
$$\chi(n,T)=\psi(n/T)$$
where $\psi:\ [0,\infty)\rightarrow[0,\infty)$ is infinitely
differentiable, nonincreasing, such that $\psi(x)=1$ if $x\in[0,1]$,
$0\le\Psi(x)\le1$ if $x\in[1,2]$, $\psi(x)=0$ if $x\ge2$.  Then
$K_T$ satisfies Assumption
\ref{asskernel}.
\end{prop}
\noindent   The delayed means kernel has the nice property that it
does not require prior knowledge of the regularity $s$ in Assumption
\ref{asskernel}.  The Dirichlet kernel satisfies \eqref{k2},
\eqref{k3} (for $p=2$) and \eqref{k4} of Assumption \ref{asskernel}.
Like the delayed means kernel, it achieves the optimal rate of
approximation without the prior knowledge of $s$.

\begin{proof}[\textupandbold{Proof of Proposition~\ref{pkernel}}]
First consider the Riesz kernel.  \eqref{k1} follows from
(2.4) in Ditzian (1998) and by the fact that
Ces\`aro kernels  $C_h^l$ are uniformly bounded in L$^1(\mathbb S^{d-1})$ for
$l > \frac {d-2}2$ (see, e.g. Bonami and Clerc 1973, p. 225).  To show \eqref{k3}
we use Theorem 4.1 in Ditzian (1998), by letting $P(D) = \Delta^S$,
$\lambda = \zeta_{T,d} +1 = T(T+d-2)+1$, $\alpha = s/2$ and $m = 1$.   Then it
implies an approximation error upper bound $CK_{s/2}(f,\Delta^S,(\zeta_{T,d} + 1)^{-\frac s 2})$,
which, in turn, is bounded by $CT^{-s}\|(-\Delta^S)^{s/2}f\|_p$ (see equations
(4.2) and (4.1) therein).  By the definition of the norm of the Sobolev space
$W_p^s(\mathbb S^{d-1})$ (see Definition \ref{def:sobolev}) the result follows.
Concerning the delayed means, \eqref{k1} follows from Theorem 2.2 and Proposition 2.5 of
Narcowich et al. (2006).
\eqref{k2} corresponds to Lemma 2.6 in Narcowich et al. (2006).
To see \eqref{k3}, use Lemma 2.4 (c) in Narcowich et al. (2006)
to obtain an upper bound $C \inf_{g \in \bigoplus_{n=0}^{T/2}H^{n,d}}\|f - g\|_p$.
Let $\lambda = \zeta_{T/2,d} + 1 = \frac T 2 (\frac T 2 + d -2) + 1$,
$\alpha = s/2, m = 1, P(D) = \Delta^S$ in Ditzian's (1998) Theorem 6.1, which
gives an upper bound on the best spherical harmonic approximation in $\xLn^p(\xSd)$
to functions in $\xWn_p^s(\xSd)$ (see also Kamzolov, 1983), then apply equation (4.1)
in Ditzian (1998) again to obtain the desired result.
\end{proof}

If the function $f$ is in $\xLtwo(\xSd)$ then Equations \eqref{e8}
and \eqref{epar} imply that $Q_{2p,d}f(x)=Q_{2p,d}f(-x)$ and
$Q_{2p+1,d}f(x)=-Q_{2p+1,d}f(-x)$ for $p \in \mathbb N$.
Consequently, the odd order terms in the condensed harmonic
expansions of $f$, $f^+$ and $f^-$ satisfy $Q_{2p+1}f^- = Q_{2p+1}f$
and $Q_{2p+1}f^+ = 0$.  Likewise, for the even order terms in the
condensed harmonic expansions of these functions $Q_{2p}f^+ =
Q_{2p}f$ and $Q_{2p}f^- = 0$ hold.   We conclude that the sum of the
odd order terms in the condensed harmonic expansion corresponds to
$f^-$ and that of the even order terms to
$f^+$.  As anticipated from
the analysis of the $d=2$ case, the operator $\mathcal H$ reduces
the even part of $f_\beta$ to a constant $\frac 1 2$, therefore
Fourier-Laplace series expansions for $f_\beta$ derived later
involve only odd order terms.

We now provide a formula that is used to obtain our estimator for
$f_\beta$.  If a non-negative function $f$ has its support included
in some
hemisphere of $\mathbb S^{d-1}$ then
\begin{equation}
f(x)=2f^-(x)\I\left\{f^-(x)>0\right\}.\label{backup}
\end{equation}
Denote the support of $f$ by $\supp f$ and let $-\supp f = \{x|-x
\in \supp  f\}$, then this formula  follows from
the fact that
$f^-(x)=f^+(x)\geq0$ on $\supp f$ while $f^-(x)=-f^+(x)\leq0$ on
$-\supp f$ and both $f^-$ and $f^+$ are $0$
on $\xSd\setminus\left(\supp f\bigcup-\supp f\right)$.\\
\begin{rem}\label{rcst}
If $f$ is a probability density function, the coefficient of degree
0 in the expansion of $f$ on surface harmonics is $1/|\xSd|$.
Conversely, any harmonic polynomial or series such that its degree
0 coefficient is $1/|\xSd|$ integrates to one.
\end{rem}
The next theorem shows that Fourier-Laplace series on the sphere is a natural
tool for the study of the operator $\mathcal H$.
\begin{thm}[Funk-Hecke Theorem]\label{t2} If $g$ belongs to $H^{n,d}$
for some $n$, and a function $F$ on $(-1,1)$ satisfies
$$\int_{-1}^1|F(t)|^2(1-t^2)^{(d-3)/2}dt<\infty,$$ then
\begin{equation}\label{e14}
\int_{\xSd}F(x'y)g(y)d\sigma(y)=\lambda_n(F)g(x)
\end{equation}
where
$$\lambda_n(F) = |\mathbb S^{d-2}|C^{\nu(d)}_n(1)^{-1}\int_{-1}^1 F(t)C^{\nu(d)}_n(t)(1-t^2)^{\frac {d-3}2}dt.$$
\end{thm}
In other words, the kernel operator defined by
\begin{equation*}
f\in\xLtwo(\xSd)\mapsto\
\left(x\mapsto\int_{\xSd}F(x'y)f(y)d\sigma(y)\right)\in\xLtwo(\xSd)
\end{equation*}
is, in the subspace $H^{n,d}$, equivalent to the multiplication by
$\lambda_n(F)$. Thus a basis of surface harmonics diagonalizes an
integral operator if its kernel is a function of the scalar product $x'y$.
\begin{rem}
Healy and Kim (1996) use Fourier-Laplace expansions to analyze a
deconvolution problem on $\mathbb S^2$. As we shall
see below, the Addition Formula
along with condensed harmonic expansions provide a general treatment
that works for arbitrary dimensions.
\end{rem}

\subsection{The Hemispherical Transform}\label{sec:hemisp}
The hemispherical transform $\mathcal H$, defined by
$\mathcal H f (x) = \int_{\mathbb S^{d-1}} \I\{x'y \geq
0\}f(y)d\sigma(y)$, plays a central role in our analysis.  It is a
special case of the operator considered in the Funk-Hecke theorem
above, with $F(t) = \I\{t \in [0,1]\}$, therefore the next
proposition follows.
\begin{notation} We define
$\lambda(n,d)=\lambda_n\left(\I\left\{t\in[0,1]\right\}\right)$ for
$d\ge3$ and $\lambda(n,2)=\frac{2\sin(n\pi/2)}{n}$.
\end{notation}
\begin{prop}\label{p3}
When $d\ge2$, the coefficients $\lambda(n,d)$ have the following
expressions
\begin{enumerate}[\textup{(}i\textup{)}]
\item\label{p3i}
$\lambda(0,d)=\frac{|\xSd|}{2}$
\item\label{p3ii}
$\lambda(1,d)=\frac{|\mathbb{S}^{d-2}|}
{d-1}$
\item\label{p3iii}
$\forall p\in\mathbb{N}\setminus\{0\},\ \lambda(2p,d)=0$
\item\label{p3iv}
$\forall p \in \mathbb N,\
\lambda(2p+1,d)=\frac{(-1)^{p}|\mathbb{S}^{d-2}|1\cdot3\cdots(2p-1)}
{(d-1)(d+1)\cdots(d+2p-1)}$.
\end{enumerate}
\end{prop}
\begin{proof}[\textupandbold{Proof of Proposition~\ref{p3}}]
Define $\alpha(n,d):=C_n^{\nu(d)}(1)|\mathbb{S}^{d-2}|^{-1}\lambda_n
\left(\I\left\{t\in[0,1]\right\}\right)$.  By the Funk-Hecke theorem
\begin{equation*}
\alpha(n,d)=\int_0^1C_n^{\nu(d)}(t)(1-t^2)^{(d-3)/2}dt,
\end{equation*}
thus using \eqref{e9},
\begin{equation*}
\alpha(n,d)=\frac{(-2)^{-n}(d-2)_n}{n!\left((d-1)/2\right)_n}\int_0^1
\frac{\xdif^n}{\xdif t^{n}}(1-t^2)^{n+(d-3)/2}dt.
\end{equation*}
Therefore for $n\geq1$ and $d\geq3$,
\begin{equation*}
\alpha(n,d)=-\frac{(-2)^{-n}(d-2)_n}{n!\left((d-1)/2\right)_n}
\left.\frac{\xdif^{n-1}}{\xdif
t^{n-1}}(1-t^2)^{n-1+(d-3)/2}dt\right|_{t=0}
\end{equation*}
since the term on the right hand-side is equal to 0 for $t=1$. To
prove that the coefficients $\alpha(2p,d)$ are equal to zero for $p$
positive it is enough to prove
\begin{equation*}
\left.\frac{\xdif^{2p+1}}{\xdif
t^{2p+1}}(1-t^2)^{2p+1+m}\right|_{t=0}=0,\quad \forall m\geq1,\
p\geq0.
\end{equation*}
The Fa\'a di Bruno formula gives that this quantity is equal to
\begin{equation*}
\left.\sum_{k_1+2k_2=2p+1}\frac{(-1)^{2p+1-k_2}(2p+1)!(m+1)\cdots(2p+1+m)}
{k_1!k_2!}(1-t^2)^{m+k_2}(2t)^{k_1}\right|_{t=0}.
\end{equation*}
and the result follows since $k_1$ in the sum cannot be equal to 0.

When $n=2p+1$ for $p \in \mathbb N$ we obtain, again using the Fa\'a
di Bruno formula, that the derivative at $t=0$ is equal to
\begin{equation*}
(-1)^{p}\frac{(2p)!}{p!}\left[(2p+1+(d-3)/2)(2p+(d-3)/2)\cdots(p+2+(d-3)/2)\right].
\end{equation*}
Together with \eqref{eval11}, the desired result follows.  For the case $d=2$ we use Proposition \ref{p2}.
\end{proof}
Define
$\xLn_{{\rm odd}}^2(\xSd)$ and $\xHn_{{\rm odd}}^s(\xSd)$ as the restrictions of $\xLtwo(\xSd)$
and $\xHn^s(\xSd)$ to odd functions and similarly
$\xLn_{{\rm even}}^2(\xSd)$ and $\xHn_{{\rm even}}^s(\xSd)$ for even functions.
The following corollary is a direct consequence of the Funk-Hecke Theorem and
Proposition \ref{p3}, and corresponds to an observation made in Remark
\ref{rsf} for the $d=2$ case.
\begin{cor}\label{c2}
The null space of the hemispherical transform $\mathcal{H}$ is given by
\begin{equation*}
\xker\ \mathcal{H}=\bigoplus_{p=1}^{\infty}H^{2p,d}=\left\{f\in\xLn_{{\rm even}}^2(\xSd):\
\int_{\xSd}f(x)d\sigma(x)=0\right\},
\end{equation*}
when $\mathcal H$ is viewed as an operator on $\xLtwo(\xSd)$.
The spaces $H^{0,d}$ and $H^{2p+1,d}$ for $p \in \mathbb N$ are the
eigenspaces associated with the non-zero eigenvalues of $\mathcal H$.
\end{cor}
As a consequence of Proposition \ref{p3}, $\mathcal{H}$ is not injective
and restrictions have to be imposed in order to ensure identification of $f_\beta$.
Section \ref{sec:deconvolution} presents sufficient conditions that allows us
to reconstruct $f_{\beta}$ from $f_{\beta}^-$.

The following proposition can be found in Rubin (1999).
\begin{prop}\label{p4}
$\mathcal{H}$ is a bijection
from $\xLn_{{\rm odd}}^2(\xSd)$ to $\xHn_{{\rm odd}}^{d/2}(\xSd)$.
\end{prop}

\begin{lem}\label{l1}
\begin{align}
h(n,d)&\asymp n^{d-2}\label{ein1},\\
|\lambda(2p+1,d)|&\asymp p^{-d/2}\label{ein4}.
\end{align}
\end{lem}
\begin{proof}
Estimate \eqref{ein1} is clearly satisfied when $d=2$ and $3$ since
$h(n,2)=2$ and $h(n,3)=2n+1$. When $d\ge4$ we have
$$h(n,d)=\frac{2}{(d-2)!}(n+(d-2)/2)[(n+1)(n+2)\cdots(n+d-3)],$$
and the results follow.\\
Next we turn to \eqref{ein4}.   When $d$ is even and $p\ge d/2$
$$|\lambda(2p+1,d)|=\frac{\kappa_d}{(2p+1)(2p+3)\cdots(2p+d-1)}$$
where
$$\kappa_d=\frac{|\mathbb{S}^{d-2}|1\cdot3\cdots(d-1)}{d-1}$$
and \eqref{ein4} follows.
Sterling's double inequality (see Feller (1968) p.50-53), that is,
\begin{equation*}
\sqrt{2\pi}n^{n+1/2}\exp\left(-n+\frac{1}{12n+1}\right)<n!<\sqrt{2\pi}
n^{n+1/2}\exp\left(-n+\frac{1}{12n}\right),
\end{equation*}
implies that
$$\frac{(2^pp!)^2}{(2p)!}\asymp \sqrt{p}$$ and therefore
$$1\cdot3\cdots(2p-1)\asymp \sqrt{p}2\cdot4\cdots(2p).$$
Thus for $p\ge d/2$ and $d$ odd we have
$$|\lambda(2p+1,d)|\asymp\frac{\sqrt{p}}{(2p+2)(2p+4)\cdots(2p+d-1)}$$
and \eqref{ein4} holds for both even and odd $d$.
\end{proof}
We can now easily check that
\begin{prop}\label{p4n}
For all $s > 0$, there exists positive constants $C_l$ and $C_u$ such that
for all $f$ in $\xHn^{s}(\xSd)$
$$C_l\left\|f^-\right\|_{2,s}\le\left\|\mathcal{H}(f^-)\right\|_{2,s+d/2}\le C_u\left\|f^-\right\|_{2,s}.$$
\end{prop}
\begin{proof}[\textupandbold{Proof of Proposition~\ref{p4n}}]
By definition we have $$\|\mathcal{H}\left(f^-\right)\|_{2,s+d/2}^2=
\sum_{p=0}^{\infty}(1+\zeta_{2p+1,d})^{s+d/2}\|Q_{2p+1,d}\mathcal{H}(f^-)\|_2^2$$
where according to the Funk-Hecke Theorem
\begin{align*}
Q_{2p+1,d}\mathcal{H}(f^-)&=Q_{2p+1,d}\mathcal{H}\left(\sum_{q=0}^{\infty}Q_{2q+1,d}f\right)\\
&=Q_{2p+1,d}\left(\sum_{q=0}^{\infty}\lambda(2q+1,d)Q_{2q+1,d}f\right)\\
&=\lambda(2p+1,d)Q_{2p+1,d}f.
\end{align*}
The result follows since Lemma \ref{l1} implies that $(1+\zeta_{2p+1,d})^{s+d/2}\lambda^2(2p+1,d)\asymp(1+\zeta_{2p+1,d})^{s}$.
\end{proof}
The factor $d/2$ in Proposition \ref{p4n} corresponds to the degree
of ``regularization'' due to smoothing by $\mathcal H$.
Now the
inverse of an odd function $f^-$ is given by
\begin{equation}\label{einv2}
\mathcal{H}^{-1}(f^-)(y)= \sum_{p=0}^{\infty}\frac{1}{\lambda(2p+1,d)}\int_{\mathbb S^{d-1}} q_{2p+1,d}(x,y)f^-(x)d\sigma(x).
\end{equation}
This is straightforward given our results at hand: for example, operate $\mathcal H$ on the RHS to see:
\begin{align*}
\mathcal H \left(\sum_{p=0}^\infty \frac 1 {\lambda(2p+1,d)}\int_{\mathbb S^{d-1}} q_{2p+1,d}(x,y)f^-(x)d\sigma(x)\right) & =   \sum_{p=0}^\infty \frac 1 {\lambda(2p+1,d)} \mathcal H Q_{2p+1,d} f^-
\\
&
=  \sum_{p=0}^\infty \frac {\lambda(2p+1,d)} {\lambda(2p+1,d)} Q_{2p+1,d} f^-\ \text{(by the Funk-Hecke Theorem)}
\\
&
= f^-.
\end{align*}
If $f^-$ belongs to $\xHn^{d/2}(\xSd)$, then
$\mathcal{H}^{-1}(f^-)(b)$ is a well-defined $\xLtwo(\xSd)$
function. Otherwise it should be understood as a distribution and is
only defined in a Sobolev space with negative exponent.  Moreover,
if $d$ is a multiple of 4, it is possible to relate the inverse of
the operator $\mathcal H$ with differentiation as in the case of
$d=2$:
\begin{prop}\label{pinv} If $d$ is a multiple of 4,
$$\mathcal{H}^{-1}=|\mathbb{S}^{d-2}|\prod_{k=1}^{d/4}[-\Delta^S+2(k-1)(d-2k)].$$
\end{prop}
\begin{proof}[\textupandbold{Proof of Proposition~\ref{pinv}}]
If we consider the case where $d$ is even, we know from Proposition
$\ref{p3}$, that
$$\frac{1}{\lambda(2p+1,d)}=(-1)^p|\mathbb{S}^{d-2}|(2p+1)(2p+3)
\hdots(d+2p-1).$$
Thus if $d$ is a multiple of 4,
$$\frac{1}{\lambda(2p+1,d)}=|\mathbb{S}^{d-2}|\prod_{k=1}^{d/4}
[-\zeta_{2p+1,d}+2(k-1)(d-2k)].$$
Using this and \eqref{einv},
\begin{align*}
\mathcal H^{-1} &= \sum_{p=0}^\infty \frac 1 {\lambda(2p+1,d)} Q_{2p+1,d}
\\
&=  \sum_{p=0}^\infty |\mathbb{S}^{d-2}| \left(\prod_{k=1}^{d/4}
[-\zeta_{2p+1,d}+2(k-1)(d-2k)]\right) Q_{2p+1,d}.
\end{align*}
Recall \eqref{laplacian} and the proposition is proved.  \end{proof}
\noindent This connection between the inverse of $\mathcal H$ and
differentiation suggests that a Bernstein-type inequality might hold
for $\mathcal H^{-1}$.   Indeed, even though the above inversion
formula is concerned with $d$'s that are multiples of 4, the
following Bernstein inequality holds for every
dimension.
\begin{thm}[Bernstein inequality]\label{tbernstein} For
every $d\ge2$ and every $q\in[1,\infty]$, there exists a positive
constant $B(d,q)$ such that for all $P$ in
$\bigoplus_{p=0}^TH^{2p+1,d}$,
\begin{equation}\label{eBernstein2}
\|\mathcal{H}^{-1}P\|_{q}\le B(d,q)T^{d/2}\|P\|_{q}.
\end{equation}
\end{thm}
\begin{proof}[\textupandbold{Proof of Theorem~\ref{tbernstein}}]
We can write $$\mathcal{H}^{-1}=P_1(D)-P_2(D)$$
where $P_1(D)$ and $P_2(D)$ are defined for all odd function $f^-$ by
\begin{align*}
P_1(D)f^-&=\sum_{p=0}^{\infty}\frac{1}{\lambda(4p+3)}\int_{\xSd}q_{4p+3}(x,y)f^-(x)d\sigma(x)\\
P_2(D)f^-&=-\sum_{p=0}^{\infty}\frac{1}{\lambda(4p+1)}\int_{\xSd}q_{4p+1}(x,y)f^-(x)d\sigma(x).
\end{align*}
$P_1(D)$ and $P_2(D)$ are two unbounded operators on $B=\xLn_{{\rm
odd}}^q(\xSd)$ with non-positive eigenvalues. We apply Theorem 3.2.
of Ditzian (1998) to $-P_1(D)$ and $-P_2(D)$ choosing $\alpha=1$.
Condition (1.6) of Ditzian (1998) can be verified using Proposition
2.2 with $r=1$ and $p = q$ and the fact that for the Cesaro kernels
$C_h^l$ are uniformly bounded in L$^1(\mathbb S^{d-1})$ for $l >
\frac {d-2} 2$ (see, e.g. Bonami and Clerc, 1973).
We see, using the
triangle inequality, that for all $P$ in
$\bigoplus_{p=0}^TH^{2p+1,d}$,
\begin{align*}
\|\mathcal{H}^{-1}P\|_{q}&\le C\frac{1}{\lambda^2(2T + 1,d)}\|P\|_{q}\\
&\le CT^d\|P\|_{q}.
\end{align*}
The last inequality follows from \eqref{ein4}.
\end{proof}

Rubin (1999) gives other inversion formulas for the Hemispherical
transform in terms of differential operators. The fact that the
inversion roughly corresponds to differentiation is another
manifestation of the ill-posedness of our problem at hand.  The
inverse operator $\mathcal H^{-1}$ is indeed unbounded.  We call the
factor $d/2$ in (\ref{eBernstein2})
the degree of ill-posedness of
the inverse problem.  For the case $q=2$, there exists a lower bound
for $\|\mathcal H^{-1}P\|_q$ in  \eqref{eBernstein2} of order
$T^{d/2}$ as well, implying that the upper bound $T^{d/2}$ in the
order of $T$ obtained in Theorem \ref{tbernstein} is tight.

\subsection{Estimators for the Choice Probability Function}\label{sec:ChoiceProbability}
This section considers estimation of the choice probability function
$r$ and its extension $R$.  We propose an estimator for $r$, which,
in turn, yields a computationally simple estimator for $f_\beta$.
Also the asymptotic results presented here are useful for the next
section where we study the limiting properties of our estimator for
the random coefficients density $f_\beta$.

Since $R$ is square integrable on $\xSd$, it has a condensed
harmonic expansion which enables us to obtain the expressions in the
next theorem.
\begin{thm}\label{t11} For $x$ in $\xSd$, we have
\begin{equation}\label{e18}
R(x)=\frac12+\sum_{p=0}^{\infty}
\xE\left[\frac{(2Y-1)}{f_X(X)} q_{2p+1,d}(X,x)\right].
\end{equation}
\end{thm}
\noindent This suggests an estimator of the form
$\hat{R}_1(x)=\frac12+\hat{R}^-_1$
with
$$\hat{R}^{-}_1(x)=\frac{1}{N}\sum_{i=1}^N\frac{(2y_i-1)}
{\hat{f}_X(x_i)}\sum_{p=0}^{T_N}q_{2p+1,d}(x_i,x)$$
where $\hat{f}_X$
is an estimator of $f_X$ and $T_N$ is a suitably chosen sequence diverging to infinity with $N$.  Note that
the second summation corresponds to the Dirichlet kernel.
We can generalize this, by introducing a class of estimators of the form
\begin{equation}\label{er}
\hat{R}^{-}_2(x)=\frac{1}{N}\sum_{i=1}^N\frac{(2y_i-1)}
{\hat{f}_X(x_i)}K_{2T_N}^-(x_i,x)
\end{equation}
where $K_{2T_N}^-$ is the odd part of a kernel of the form
\eqref{ekernel} satisfying Assumption \ref{asskernel}, such as the
two kernels in Proposition \ref{pkernel}.

The estimator \eqref{er} is convenient, though the plug-in term
$\hat{f}_X$ has to be treated with care.  We avoid restrictive
assumptions on the distributions of covariates and allow $f_X(x)$ to
decay to zero as $x$ approaches the boundary of its support $H^+$.
To deal with the latter problem, we modify \eqref{er} by
\begin{equation}\label{ert}
\hat{R}^{-}(x)=\frac{1}{N}\sum_{i=1}^N\frac{(2y_i-1)K_{2T_N}^{-}(x_i,x)}
{\max\left(\hat{f}_X(x_i),m_N\right)}
\end{equation}
where $m_N$ is a trimming factor going to 0 with the sample size.
Our estimator for $R$ is then
\begin{equation}\label{eestRbis}
\widehat{R}=\frac12+\hat{R}^{-}.
\end{equation}
\begin{rem}
Alternative estimators of $R^-$ are available.  For example, one may
use kernel regression on the sphere to estimate $r$ in order to
obtain an estimator for $R^-$.   As noted before, however, we then
need to use numerical integration to evaluate (\ref{eestb2}) to
calculate $\hat f_\beta^-$.
\end{rem}

\begin{proof}[\textupandbold{Proof of Theorem~\ref{t11}}] $R$ has the
following condensed harmonic expansion
$$R(x)=\frac{1}{2}+\sum_{p=1}^{\infty}(Q_{2p+1,d}R)(x).$$ We then
write using \eqref{e2bb}, changing variables and using \eqref{epar},
\begin{align*}
(Q_{2p+1,d}R)(x)&=\int_{\xSd}q_{2p+1,d}(x,z)R(z)d\sigma(z)\\
&=\int_{H^+}q_{2p+1,d}(x,z)r(z)d\sigma(z)+\int_{H^-}q_{2p+1,d}(x,z)(1-r(-z))d\sigma(z)\\
&=\int_{H^+}q_{2p+1,d}(x,z)r(z)d\sigma(z)-\int_{H^+}q_{2p+1,d}(x,z)(1-r(z))d\sigma(z)\\
&=\int_{H^+}q_{2p+1,d}(x,z)(2r(z)-1)d\sigma(z)\\
&=\int_{H^+}q_{2p+1,d}(x,z)\xE\left[\left.\frac{2Y-1}{f_X(z)}\right|X=z\right]f_X(z)d\sigma(z)\\
&=\xE\left[\frac{(2Y-1)q_{2p+1,d}(x,X)}{f_X(X)}\right].
\end{align*}
\end{proof}


\subsection{Proofs of Main Results}\label{sec:mainproofs}

\begin{proof}[\textupandbold{Proof of Proposition~\ref{pQ2}}]
It is straightforward that the model \eqref{emod} and Assumption
\ref{ass1} imply that the choice probability function $r$ given by
\eqref{e1} is homogeneous of degree 0. Proposition \ref{p4} along
with the fact that $R=\frac12+\mathcal{H}\left(f_{\beta}^-\right)$
with $f_{\beta}^-\in\xLn_{\mathrm{odd}}^2(\xSd)$ implies that $R$
belongs to $\xHn^{d/2}(\xSd)$.   We now turn to the proof of
sufficiency.   If the extension $R$ given by \eqref{e2bb} belongs
to
$\xHn^{d/2}(\xSd)$ then so does $R^-$ and Proposition \ref{p4}
shows that there exists a unique odd function $f^-$ in
$\xLtwo(\xSd)$ such that
$$
R=\frac12+\mathcal{H}\left(f^-\right)=\mathcal{H}\left(\frac{1}{|\xSd|}+f^-\right).
$$
Moreover, since $0\le R(x)\le1$ holds for every $x \in \mathbb
S^{d-1}$, the above relationship implies that $\frac 1 2 \geq
\mathcal H f^-(x), \forall x \in \mathbb S^{d-1}$.  But $\mathcal H
f^-(x) \geq \int_{\{f^-(b)\ge0\}}f^-(b)d\sigma(b)$ holds for some
$x$.  Therefore we conclude that $\frac 1 2 \geq
\int_{\{f^-(b)\ge0\}}f^-(b)d\sigma(x)
=-\int_{\{f^-(b)\le0\}}f^-(b)d\sigma(b)$,
thus $\int_{\xSd}|f^-(b)|d\sigma(b)\le1$. Also, following the
discussion in Section \ref{sec:spheres}, $\frac{1}{|\xSd|}+f^-$
integrates to 1.  We have seen in Corollary \ref{c2} that for even
function $g$ that has 0 as the coefficient of degree 0 in its
expansion on the surface harmonics (i.e. an even function that
integrates to zero over the sphere),
$$R= \mathcal{H}\left(g+\frac{1}{|\xSd|}+f^-\right)$$
holds.  Now consider
$$g=|f^-|-\frac{1}{|\xSd|}\int_{\xSd}|f^-(b)|d\sigma(b),$$
then this certainly is even and integrates to zero.   Using this, define
$$f_{\beta}^*:=g+\frac{1}{|\xSd|}+f^-=2f^-\I\{f^->0\}+\frac{1}{|\xSd|}
\left(1-\int_{\xSd}|f^-(b)|d\sigma(b)\right)\ge0.$$
Obviously ${f_{\beta}^*}^- = f^-$. This function $f_{\beta}^*$ is non-negative and integrates to one,
and thus it is a proper probability density function (pdf).
It is indeed bounded from below by
$\frac{1}{|\xSd|}\left(1-\int_{\xSd}|f^-(b)|d\sigma(b)\right)$.  As a consequence, there exists
a pdf $f_{\beta}^*$ such that
$$R=\mathcal{H}\left(f_{\beta}^*\right)=\frac12+\mathcal{H}\left({f_{\beta}^*}^-\right)$$
and for all $x$ in $H^+$,
$r(x)=\mathcal{H}\left(f_{\beta}^*\right)(x)$.
\end{proof}

\begin{proof}[\textupandbold{Proof of Theorem~\ref{f_beta_rate}}]
We use the shorthand notation $\I(b) :=  \I\{f_\beta^-(b) > 0\}$
and  $\hat \I(b) := \I\{\hat f_\beta^-(b) > 0\}$.  Then
$f_\beta = 2f^-_\beta\I$ and $\hat f_\beta = 2 \hat
f^-_\beta \hat \I$.  We write
\begin{align*}
\overline{f}_{\beta,T}^{\ -}(b)&=\frac{1}{N}\sum_{i=1}^N\frac{(2y_i-1)\mathcal{H}^{-1}\left(K_{2T_N}^{-}(x_i,\cdot)\right)(b)}
{\max\left(f_X(x_i),m_N\right)}\\
\overline{f}_{\beta}^{\ -}(b)&=\frac{1}{N}\sum_{i=1}^N\frac{(2y_i-1)\mathcal{H}^{-1}\left(K_{2T_N}^{-}(x_i,\cdot)\right)(b)}
{f_X(x_i)}
\end{align*}
\noindent and use the decomposition
\begin{equation}\label{edec}
\hat{f}_{\beta}^--f_{\beta}^-=\left(\hat{f}_{\beta}^-
-\overline{f}_{\beta,T}^{\ -}\right)+\left(\overline{f}_{\beta,T}^{\ -}
-\xE\left[\overline{f}_{\beta,T}^{\ -}\right]\right)+\left(\xE
\left[\overline{f}_{\beta,T}^{\ -}\right]-\xE\left[\overline{f}_{\beta}^{\ -}\right]\right)
+\left(\xE\left[\overline{f}_{\beta}^{\ -}\right]-f_{\beta}^{-}\right),
\end{equation}
and denote the terms on the right hand side by $S_{\mathrm {p}}$ (stochastic component due to plug-in),
$S_{\mathrm e}$ (stochastic component of the infeasible estimator $\overline{f}_{\beta,T}^{\ -}$), $B_{\mathrm t}$ (trimming bias)
and $B_{\mathrm a}$ (approximation bias).

Take $q\in[1,\infty)$,
\begin{align*}
\|\hat f_\beta - f_\beta\|_q^q
= &\int (\hat f_{\beta}(b) - f_{\beta}(b))^q d\sigma(b)
\\
= & \int_{\mathbb I(b) = 1, \hat {\mathbb I} (b) = 1} (\hat f_{\beta}(b) - f_{\beta}(b))^q d\sigma(b)
+ \int_{\mathbb I(b) = 0, \hat {\mathbb I}(b) = 1} (\hat f_{\beta}(b) - f_{\beta}(b))^q d\sigma(b)
\\
&+ \int_{\mathbb I(b) = 1, \hat {\mathbb I}(b) = 0} (\hat f_{\beta}(b) - f_{\beta}(b))^q d\sigma(b)
+ \int_{\mathbb I(b) = 0, \hat {\mathbb I}(b) = 0} (\hat f_{\beta}(b) - f_{\beta}(b))^q d\sigma(b)
\\
:= & A_1 + A_2 + A_3 + A_4.
\end{align*}
Obviously
$$
A_1 = \int_{\mathbb I(b) = 1, \hat {\mathbb I}(b) = 1} (2\hat f^-_{\beta}(b) - 2f^-_{\beta}(b))^q d\sigma(b)
$$
and $A_4 =0$.
Also,
$$
A_2 = \int_{\mathbb I(b) = 0, \hat {\mathbb I}(b) = 1} (2\hat f^-_{\beta}(b) - f_{\beta}(b))^q d\sigma(b).
$$
But given ${\mathbb I}(b) = 0$ and $\hat {\mathbb I}(b) = 1$, $2\hat f^-_{\beta}(b) > 0$,
$f_\beta(b) = 0$ and $2 f^-_{\beta}(b) \leq 0$, so replacing
$f_\beta$ with $2f^-_\beta$ in the bracket,
$$
A_2 \leq  \int_{\mathbb I(b) = 0, \hat {\mathbb I}(b) = 1} (2\hat f^-_{\beta}(b)
- 2f^-_{\beta}(b))^q d\sigma(b).
$$
Similarly,
$$
A_3 = \int_{\mathbb I(b) = 1, \hat {\mathbb I}(b) = 0} (\hat f_{\beta}(b) -
2f^-_{\beta}(b))^q d\sigma(b).
$$
and  given ${\mathbb I}(b) = 1$ and $\hat {\mathbb I}(b) = 0$, $2f^-_{\beta}(b) > 0$,
$\hat f_\beta(b) = 0$ and $2\hat  f^-_{\beta}(b) \leq 0$, so replacing
$f_\beta$ with $2f^-_\beta$ in the bracket,
$$
A_3 \leq  \int_{\mathbb I(b) = 0, \hat {\mathbb I}(b) = 1} (2\hat f^-_{\beta}(b) -
2f^-_{\beta}(b))^q d\sigma(b).
$$
Overall,
$$
\|\hat f_\beta - f_\beta\|_q^q \leq 2^q \|\hat f^-_\beta - f^-_\beta\|_q^q.
$$
A similar proof can be carried out replacing $\xLn^q(\xSd)$ by
$\xLinfty(\xSd)$.  Thus it is enough to consider the behavior of
$\hat f^-_\beta - f^-_\beta$ instead of $\hat f_\beta - f_\beta$.
As noted above, the former can be decomposed into four terms,
$S_{\mathrm p}$, $S_{\mathrm e}$, $B_{\mathrm t}$ and $B_{\mathrm
a}$.

We start with the analysis of $S_{\mathrm p}$. Note that for
$q\in[1,\infty]$
\begin{align*}
\|S_{\mathrm p}\|_{q}&=\left\|\mathcal{H}^{-1}\left(\frac{1}{N}\sum_{i=1}^N\frac{(2y_i-1)K_{2T_N}^{-}
(x_i,\cdot)}{\max(f_X(x_i),m_N)}\left(\frac{\max\left(f_X(x_i),m_N\right)}
{\max\left(\hat{f}_X(x_i),m_N\right)}-1\right)\right)\right\|_q\\
&\le B(d,q)T_N^{d/2}\left\|\frac{1}{N}\sum_{i=1}^N\frac{(2y_i-1)K_{2T_N}^{-}
(x_i,\cdot)}{\max(f_X(x_i),m_N)}\left(\frac{\max\left(f_X(x_i),m_N\right)}
{\max\left(\hat{f}_X(x_i),m_N\right)}-1\right)\right\|_q\quad{\rm (by\ Theorem\ \ref{tbernstein})}\\
&\le B(d,q)T_N^{d/2}m_N^{-1}\left\|\frac{1}{N}\sum_{i=1}^N\left|K_{2T_N}
(x_i,\cdot)\right|\right\|_q\max_{i=1,\hdots,N}\left|\frac{\max\left(f_X(x_i),m_N\right)}
{\max\left(\hat{f}_X(x_i),m_N\right)}-1\right|\\
&\le B(d,q)T_N^{d/2}m_N^{-2}\left\|\frac{1}{N}\sum_{i=1}^N\left|K_{2T_N}
(x_i,\cdot)\right|\right\|_q\max_{i=1,\hdots,N}\left|f_X(x_i)
-\hat{f}_X(x_i)\right|
\end{align*}
\noindent holds, where we have used the triangle inequality.  The $\xLn^q$-norm
on the right hand side is bounded from above by
\begin{equation}\label{e2terms}
\left\|\frac{1}{N}\sum_{i=1}^N\left|K_{2T_N}(x_i,\cdot)\right|
-\xE\left|K_{2T_N}(X,\cdot)\right|\right\|_q+
\left\|\xE\left|K_{2T_N}(X,\cdot)\right|\right\|_q:=\|T_1\|_q+\|T_2\|_q.
\end{equation}
First consider the term $\|T_1\|_q$.  We begin with the case of $q\in[1,2]$.
By the H\"older inequality,
\begin{align*}
\xE\left[\|T_1\|_q^q\right]&=\int_{\xSd}\xE\left[T_1(x)^q\right]d\sigma(x)\\
&\le\int_{\xSd}\xE\left[T_1(x)^2\right]^{q/2}d\sigma(x)
\end{align*}
where
\begin{align}\label{KL2bound}
\xE\left[T_1(x)^2\right]&\le\frac1N\xE\left[\left(K_{2T_N}(X,x)\right)^2\right]\\
&\le\frac{C}{N}\left\|K_{2T_N}(\star_2,x)\right\|_2^2\quad (\text{boundedness assumption on } f_X) \nonumber \\
&=\frac{C}{N}\left\|\sum_{n=0}^{2T_N}\chi(n,2T_N)q_{n,d}(\star_2,x)\right\|_2^2 \nonumber \\
&\leq\frac{C}{N} \sum_{n=0}^{2T_N} \left\|q_{n,d}(\star_2,x)\right\|_2^2
\quad (\text{by Assumption \ref{asskernel}(\ref{k4})})
\nonumber \\
&\leq\frac{C}{N} \sum_{n=0}^{2T_N}  \frac {h^2(n,d)\left\|C_n^{\nu(d)}(\star_2'x)\right\|_2^2}{|\mathbb S^{d-1}|^2(C_n^{\nu(d)}(1))^2}
\nonumber \\
&\leq\frac{C}{N} \sum_{n=0}^{2T_N}h(n,d)\quad{\rm (by\ \eqref{emass})}   \nonumber  \\
&\le\frac{CT_N^{d-1}}{N}\quad{\rm (by\ Lemma\ \ref{l1})} \nonumber.
\end{align}
By the Markov inequality,
\begin{equation}\label{l2rate}
T_N^{d/2}m_N^{-2}\|T_1\|_q=O_p\left(m_N^{-2}N^{-1/2}T_N^{(2d-1)/2}\right),
\end{equation}
providing a convergence rate for $\|T_1\|_q, q \in [1,2]$.   So if we can establish
a similar rate for $\|T_1\|_\infty$, all $\xLn^q(\xSd)$ convergence rates of $T_1$
for $q\in(2,\infty]$ can be interpolated between the $\xLtwo(\xSd)$ and
$\xLinfty(\xSd)$ convergence rates using the following inequality:
\begin{equation}\label{holder}
\forall f\in\xLinfty(\xSd),\ \|f\|_q \leq \|f\|_2^{2/q}\|f\|_{\infty}^{1-2/q}.
\end{equation}
To see this, note
\begin{align*}
\|f\|_q &= \|f^2 |f|^{q-2}\|_1^{1/q}\\
&\leq \left[\|f^2\|_1 \||f|^{q-2}\|_\infty\right]^{1/q} \quad \text{(by H\"older)}\\
&= \|f\|_2^{2/q}\|f\|_{\infty}^{1-2/q}.
\end{align*}
We can thus focus on $\|T_1\|_\infty$. We cover the sphere $\mathbb S^{d-1}$ by
$\mathfrak{N}(N,r,d)$ geodesic balls (caps)
$\left(B_i\right)_{i=1}^{\mathfrak{N}(N,r,d)}$ of centers
$\left(\tilde{x}_i\right)_{i=1}^{\mathfrak{N}(N,r,d)}$
and radius $R(N,r,d)$, that is, $B_i = \{x\in\xSd:\ \|x - \tilde{x}_i\| \le R(N,r,d)\}$.
As the notation suggests, we let the radius of the balls depend on $N$, $r$ and $d$, as
specified more precisely below.  Note that
$\mathfrak{N}(N,r,d)\asymp R(N,r,d)^{-(d-1)}$.

We now prove that for
every $\epsilon > 0$ positive, there exists a positive $M$ such that
\begin{equation}\label{t1prob}
\xP\left({v}_N T_N^{d/2}m_N^{-2}\sup_{x\in\xSd}|T_1(x)|\ge M\right)\le\epsilon
\end{equation}
holds for an appropriately chosen sequence ${v}_N \uparrow \infty$.  Write
\begin{align}\label{supbound}
&\xP\left({v}_N T_N^{d/2}m_N^{-2}\sup_{x\in\xSd}|T_1(x)|\ge M\right)\\
&\le\xP\left(\bigcup_{i=1,\hdots,\mathfrak{N}(N,r,d)}\left\{{v}_N
T_N^{d/2}m_N^{-2}|T_1(\tilde{x}_i)|\ge M/2\right\}\right) \nonumber \\
&\ \ \ +\xP\left(\exists i\in \{1,\hdots,\mathfrak{N}(N,r,d)\}:\ {v}_N
T_N^{d/2}m_N^{-2}\sup_{x\in B_i}|T_1(x)-T_1(\tilde{x}_i)|\ge M/2\right) \nonumber \\
&\le\mathfrak{N}(N,r,d)\sup_{i=1,\hdots,\mathfrak{N}_N}\xP\left({v}_N
T_N^{d/2}m_N^{-2}|T_1(\tilde{x}_i)|\ge M/2\right) \nonumber
\end{align}
where the last inequality is obtained using Assumption
\ref{asskernel} \eqref{k2} on the kernel and letting
$R(N,r,d)\asymp m_N^2{v}_N^{-1}T_N^{-(d/2 + \alpha)}M$ (where
$\alpha$ is given in Assumption \ref{asskernel} \eqref{k2}). Notice
\begin{align}\label{bernsteinbound}
&\xP\left({v}_N T_N^{d/2}m_N^{-2}|T_1(\tilde{x}_i)|\ge M/2\right)\\
&=\xP\left(\left|\sum_{j=1}^N\frac{\left|K_{2T_N}(x_j,\tilde{x}_i)\right|}{T_N^{d-1}}
-\xE\left[\frac{\left|K_{2T_N}(X,\tilde{x}_i)\right|}{T_N^{d-1}}\right]\right|\ge
T_N^{-(d-1)}{v}_N^{-1}T_N^{-d/2}m_N^2NM/2\right) \nonumber \\
&\le2\exp\left\{-\frac12\left(\frac{t^2}{\omega+Lt/3}\right)\right\}\quad{\rm(Bernstein\ inequality)} \nonumber
\end{align}
where
\begin{align*}
t&=T_N^{-(d-1)}{v}_N^{-1} T_N^{-d/2}m_N^2NM/2\\
\omega&\ge\sum_{j=1}^N{\rm var}\left(\frac{\left|K_{2T_N}(X_j,\tilde{x}_i)\right|}{T_N^{d-1}}\right)\\
&\forall j=1,\hdots,N,\ \left|\frac{K_{2T_N}(X_j,\tilde{x}_i)}{T_N^{d-1}}\right| \le
L\quad {\rm (using \ \eqref{e8}\ and\ \eqref{einfty})}.
\end{align*}
The bound $L$ in the last line is obtained by noting that
$\left|K_{2T_N}(X_j,\tilde{x}_i)\right| =  \left|\sum_{n = 0}^{2T_N}
\chi(n,2T_N)q_{n,d}(X_j,\tilde{x}_i) \right | \leq C \sum_{n=0}^{2T_N}|h(n,d)|
\asymp T_N^{d-1}$, which follows from \eqref{e8}, \eqref{einfty} and \eqref{ein1}.
Here we can take $\omega = CN \mathbb E[K_{2T_N}(X,\tilde x_i)^2]$, then
by the calculations in \eqref{KL2bound}, we can write $\omega=CNT_N^{-(d-1)}$.
$\omega$ is the leading term
in the denominator of the exponent in the last inequality.

If we take ${v}_N=(\log N)^{-1/2}m_N^2N^{1/2}T_N^{-(2d-1)/2}$, then
\begin{equation}\label{exprate}
\frac {t^2}{\omega + Lt/3} \asymp (\log N) M^2.
\end{equation}
Also, use this ${v}_N$
in our choice of $R(N,r,d)$ made above to get:
$$
R(N,r,d) \asymp m_N^2{v}_N^{-1}T_N^{-(d/2 + \alpha)}M =
(\log(N))^{1/2} N^{-1/2} T_N^{\frac{d-1}2 - \alpha}M
$$
Thus
\begin{equation}\label{ballnumber}
\mathfrak{N}(N,r,d)\asymp R(N,r,d)^{-(d-1)} = \exp\left(C_1\log N  + o(\log N)\right)
\end{equation}
for some constant $C_1$ that might be greater than $\frac 1 2 (d-1)$, depending on the value of $\alpha$.
Indeed, $T_N$ does not grow more than polynomially fast in $N$.
\eqref{supbound}, \eqref{bernsteinbound},
\eqref{exprate} and \eqref{ballnumber} imply that, for a positive constants $C$ and $C_2$,
\begin{equation}\label{eexpbound}
\xP\left({v}_N T_N^{d/2}m_N^{-2}\sup_{x\in\xSd}|T_1(x)|\ge
M\right)\le C\exp\left\{(\log N)(C_1-C_2M^2)\right\}
\end{equation}
holds.  For a large enough $M$, $C_1 - C_2M^2 < 0$ and the right hand side of
\eqref{eexpbound} converges to zero, so \eqref{t1prob} follows.  In summary,
we have just shown that
$$
 T_N^{d/2}m_N^{-2}\|T_1\|_{\infty}=O_p\left((\log N)^{1/2}m_N^{-2}N^{-1/2}T_N^{(2d-1)/2}\right)
$$
and with \eqref{l2rate} and \eqref{holder} we also conclude that
$$
T_N^{d/2}m_N^{-2}\|T_1\|_{q}=O_p\left((\log N)^{1/2-1/q}m_N^{-2}N^{-1/2}T_N^{(2d-1)/2}\right).
$$
Concerning $\|T_2\|_q$, $q \in [1,\infty]$, since $f_X$ is bounded by assumption,
there exists a positive $C$ such that
$$\|T_2\|_q \leq  C\left\|\left\|K_{2T_N}(\star_1,\star_q)\right\|_1\right\|_q$$
where integration in $\|\cdot\|_1$ is with respect to argument $\star_1$
and integration in $\|\cdot\|_q$ is with respect to $\star_q$.
But $\left\|K_{2T_N}(\star_1,\star_q)\right\|_1$ is a constant and does not
depend on $\star_q$, as previously noted.  Thus
$$\left\|\left\|K_{2T_N}(\star_1,\star_q)\right\|_1\right\|_{q}=|\xSd|^{1/q}
\left\|K_{2T_N}(\star_1,\star_q)\right\|_1$$
and we conclude that this term is $O(1)$ using Assumption \ref{asskernel} \eqref{k1}
on the kernel, thus
$$T_N^{d/2}m_N^{-2}\|T_2\|_q=O\left(m_N^{-2}T_N^{d/2}\right).$$

Analogously to our treatment of $\|T_1\|_q$, we can prove that when $q\in[1,2]$,
\begin{equation*}
\left\|S_{\mathrm e}\right\|_q=O_p\left(m_N^{-1}N^{-1/2}T_N^{(2d-1)/2}\right),
\end{equation*}
while for
$q\in(2,\infty]$
\begin{equation*}
\left\|S_{\mathrm e}\right\|_{q}=O_p\left(m_N^{-1}(\log N)^{1/2-1/q}N^{-1/2}T_N^{(2d-1)/2}\right).
\end{equation*}
Let us now turn to the bias term induced by trimming
{\small\begin{align*}
B_{\mathrm t}(b)&=\xE\left[\frac{(2Y-1)\mathcal{H}^{-1}\left(K_{2T_N}^{-}(X,\cdot)\right)(b)}{f_X(X)}
\left(\frac{f_X(X)}{\max(f_X(X),m_N)}-1\right)\right]\\
&=\int_{\{z\in\xSd:\ f_X(z)<m_N\}}\xE[2Y-1|X=z]\mathcal{H}^{-1}\left(K_{2T_N}^{-}(z,\cdot)\right)(b)
\left(f_X(z)m_N^{-1}-1\right)d\sigma(z).
\end{align*}}
\noindent This yields
 {\small\begin{align*}
\left|B_{\mathrm t}(b)\right|
& \leq \int_{\xSd}\left|\mathcal{H}^{-1}\left(K_{2T_N}^{-}(z,\cdot)\right)(b)\right|
\mathbb{I}\left\{z\in\xSd:\ f_X(z)<m_N\right\}
d\sigma(z)\\
&=\int_{\xSd}\left|\mathcal{H}^{-1}\left(K_{2T_N}^{-}(b,\cdot)\right)(z)\right|
\mathbb{I}\left\{z\in\xSd:\ f_X(z)<m_N\right\}
d\sigma(z) \quad \text{(using the condensed Harmonic expansion)},
\end{align*}}
\noindent thus, for every $1\le r\le q$,
{\small\begin{align*}
\|B_{\mathrm t}\|_q&\le \left\|\mathcal{H}^{-1}\left(K_{2T_N}^{-}(b,\cdot)\right)\right\|_{r}
\sigma\left(f_X<m_N\right)^{1/q-1/r+1} \quad \text{(from Proposition \ref{pYoung})}\\
&\le C B(d,r)T_N^{d/2+(d-1)(1-1/r)}\sigma\left(f_X<m_N\right)^{1/q-1/r+1}
\end{align*}}
\noindent where in the last inequality we use Theorem \ref{tbernstein} and calculate an upper bound on the
$\xLn^r$-norm of the kernel by interpolation, using H\"older's inequality,
between the uniformly bounded $\xLone$-norm and the upper bound on the sup norm of the order
of $T_N^{d-1}$ seen previously, $C$ is a constant.
We finally treat $B_{\mathrm a}$ using Assumption \ref{asskernel} \eqref{k3} with the condition
that $f_{\beta}^-\in\xWn^s_{q}(\xSd)$:
$$\|B_{\mathrm a}\|_{q}\le CT_N^{-s}.$$
In the case where $f_X\ge m\ \sigma\ a.e.$, we use the decomposition
\begin{align*}
\hat{f}_{\beta}^--f_{\beta}^-&=\left(\hat{f}_{\beta}^-
-\overline{f}_{\beta}^{\ -}\right)+\left(\overline{f}_{\beta}^{\ -}
-\xE\left[\overline{f}_{\beta}^{\ -}\right]\right)
+\left(\xE\left[\overline{f}_{\beta}^{\ -}\right]-f_{\beta}^{-}\right)\\
&=\tilde{S}_{\mathrm p}+\tilde{S}_{\mathrm e}+B_{\mathrm a}.
\end{align*}
Now for example,
$$\|\tilde{S}_{\mathrm p}\|_{q}\le B(d,q)T_N^{d/2}\left\|\frac{1}{N}\sum_{i=1}^N\left|K_{2T_N}
(x_i,\cdot)\right|\right\|_q\frac{\max_{i=1,\hdots,N}\left|f_X(x_i)
-\hat{f}_X(x_i)\right|}{\min_{i=1,\hdots,N}|\hat{f}_X(x_i)|},$$
because $\hat{f}_X$ is a consistent estimator in sup norm,
$$\forall \epsilon>0,\ \exists N_0>0:\ \forall n\ge N_0,\
\mathbb{P}\left(\min_{i=1,\hdots,N}|\hat{f}_X(x_i)|>\frac{m}{2}\right)\le\frac{\epsilon}{2},$$
and we can treat the terms $\tilde{S}_{\mathrm p}$ and
$\tilde{S}_{\mathrm e}$ on this event.
\end{proof}

\begin{proof}[\textupandbold{Proof of the corollaries~\ref{f_beta_rate},
\ref{f_beta_rate2} and \ref{f_beta_rate3}}]
The rate $\gamma s$ in Corollary \ref{f_beta_rate}
comes from the fact that it coincides with the maximum of
\begin{equation}\label{ecomp}
\min\left(\gamma s,-\gamma\frac d2-\rho+\frac12-\gamma\frac{d-1}{2},
-\gamma\frac d2+r_X-2\rho,-\gamma\frac d2+\rho\tau-\gamma(d-1)(1-1/q)\right).
\end{equation}
for $r_X/2\le\rho<1/2$ and $0<\gamma<1/(d-1)$ which is what we get from
\eqref{eCsimp} and \eqref{eUB}.  Indeed, it is enough
to find $\gamma(\rho)$ as the minimum of
\begin{equation}\label{ecomp}
\min\left(\gamma \left(s+\frac{d}{2}\right),-\rho+\frac12-\gamma\frac{d-1}{2},
r_X-2\rho,\rho\tau-\gamma(d-1)(1-1/q)\right).
\end{equation}
The first is an increasing function of $\gamma$ while the second and fourth are decreasing.
The rest follows by simple computations.  The proofs of the convergence in probability on
Corollaries \ref{f_beta_rate2} and \ref{f_beta_rate3} is similar and simpler because
there is only one parameter $\gamma$.
In order to prove the
strong uniform consistency in Corollary \ref{f_beta_rate}, noticing that the
bias terms $B_{\mathrm t}$ and $B_{\mathrm a}$ are not stochastic
and bounded after proper scaling, we just have to focus on
$S_{\mathrm p}$ and $S_{\mathrm e}$ appearing in the proof of Theorem \ref{f_beta_ub}.
Concerning $S_{\mathrm p}$,
proceed as before and note that taking $M$ large enough so that
$C_1-C_2M^2<-1$ implies summability of the left hand side in
\eqref{eexpbound}.  We conclude from the first Borel-Cantelli lemma
that the probability that the events occur infinitely often is zero
thus with probability one
$$\overline{\lim}_{N\rightarrow\infty}{v}_N^{-1}B(d,\infty)
T_N^{d/2}m_N^{-2}\sup_{x\in\xSd}|T_1(x)|<M.$$
The term $T_2$ is non-stochastic and its treatment in our previous analysis remains
valid, therefore we can use the same non-stochastic upper bound.  We
then use Assumption \ref{ass4} \eqref{ass4ii} instead of Assumption
\ref{ass4} \eqref{ass4i} to show almost sure uniform
boundedness of $S_{\mathrm p}$ after proper rescaling. The treatment
of $S_{\mathrm e}$ is analogous to that of $T_1$.
The proof is the same in Corollaries \ref{f_beta_rate2} and \ref{f_beta_rate3}.
\end{proof}

\begin{proof}[\textupandbold{Proof of Theorem~\ref{t7}}]
We first prove that the Lyapounov
condition holds: there exists $\delta>0$ such that for $N$ going to
infinity,
\begin{equation}\label{elya}
\frac{\xE\left[\left|Z_{N}(b)-\xE\left[Z_{N}(b)\right]\right|^{2
+\delta}\right]}{N^{\delta/2}\left({\rm
var}\left(Z_{N}(b)\right)\right)^{1+\delta/2}}\rightarrow0
\end{equation}
(see, e.g. Billingsley, 1995).  We start from deriving a lower
bound on ${\rm var}\left(Z_{N}(b)\right)$. Since
$\xE[Z_{N}(b)]$ converges to $f_{\beta}^-(b)$, it
is enough to obtain a lower bound on
{\small\begin{align*}
\xE[Z_{N}^2](b)
&=4\int_{H^+}\left(\sum_{p=0}^{T_N - 1} \chi(2p+1,2T_N) \frac{q_{2p+1,d}(z,b)}
{\max\left(f_X(z),m_N\right)\lambda(2p+1,d)}\right)^2f_X(z)
d\sigma(z)\\
&=4\int_{H^+}\left(\sum_{p=0}^{T_N - 1}\chi(2p+1,2T_N)\frac{q_{2p+1,d}(z,b)}
{\lambda(2p+1,d)}\right)^2\left(\frac{1}{f_X(z)}\I\{f_X\ge m_N\}+f_X(z)m_N^{-2}
\I\{f_X< m_N\}\right)
d\sigma(z)\\
&\ge4\frac{1}{\|f_X\|_{\infty}}\int_{H^+}
\left(\sum_{p=0}^{T_N - 1}\chi(2p+1,2T_N)
\frac{q_{2p+1,d}(z,b)}
{\lambda(2p+1,d)}\right)^2
d\sigma(z)
\\*
&\qquad \qquad \qquad \qquad \qquad \qquad  - 4\frac{1}{\|f_X\|_{\infty}}\int_{\{f_X<m_N\}}
\left(\sum_{p=0}^{T_N - 1}\chi(2p+1,2T_N)
\frac{q_{2p+1,d}(z,b)}
{\lambda(2p+1,d)}\right)^2
d\sigma(z)
\end{align*}}
With similar computations as \eqref{KL2bound}, using as well \eqref{ein4},
we know that there exists a constant $C$ such that
$$
\left\|\sum_{p=0}^{T_N - 1}\chi(2p+1,2T_N)
\frac{q_{2p+1,d}(z,\star)}
{\lambda(2p+1,d)}\right\|_2  \leq CT_N^{2d - 1},
$$
therefore using Proposition \ref{pYoung} with $p=q=r=1$ we obtain
{\small $$\xE[Z_{N}^2](b)\ge\frac{4}{\|f_X\|_{\infty}}\sum_{p=0}^{T_N - 1}\chi(2p+1,2T_N)^2
\int_{H^+}\frac{q_{2p+1,d}(z,b)^2}
{\lambda(2p+1,d)^2}
d\sigma(z)-CT_N^{2d-1}\sigma\left(f_X<m_N\right).$$}

\noindent Using Assumption \ref{asskernel} \eqref{k4},
the first term on the right hand side can be bounded from below by
$$C\sum_{p=0}^{\lfloor (T_N-1)/2\rfloor}\left\|\frac{q_{2p+1,d}(z,b)}
{\lambda(2p+1,d)}\right\|_{2}^2$$
i.e. by
$CT_N^{2d-1}$.  Thus as $m_N$ decays to zero,
$\sigma\left(f_X<m_N\right)$ decays to zero and
\begin{equation}\label{lyapounov1}
\xE[Z_{N}^2](b)\ge CT_N^{2d-1}.
\end{equation}
We now derive an upper bound of
$\xE\left[\left|Z_{N}(b)\right|^{2+\delta}\right]$ using Theorem \ref{tbernstein} and
interpolation between $\xLinfty(\xSd)$ and $\xLone(\xSd)$ norms of the kernels using
the H\"older inequality:
\begin{align*}
\xE\left[\left|Z_{N}\right|^{2+\delta}\right]&\le\|f_X\|_{\infty}m_N^{-(2+\delta)}
\left\|\mathcal{H}^{-1}\left(K_{2T_N}^{-}(z,\cdot)\right)\right\|_{2+\delta}^{2+\delta}\\
&\le \|f_X\|_{\infty}m_N^{-(2+\delta)}B(d,2+\delta)^{2+\delta}T_N^{d(2+\delta)/2}
\left\|K_{2T_N}^{-}(z,\cdot)\right\|_{2+\delta}^{2+\delta}\\
&\le C m_N^{-(2+\delta)}T_N^{d(2+\delta)/2}
T_N^{(d-1)(1+\delta)}.
\end{align*}
By this and \eqref{lyapounov1} an upper bound for the ratio appearing in
\eqref{elya} is given by
$$m_N^{-(2+\delta)}\left(\frac{T_N^{d-1}}{N}\right)^{\delta/2}.$$
Therefore the Lyapounov condition is satisfied if \eqref{condLyap}
holds, and it follows that $N^{1/2}s_N^{-1}(b)S_{\mathrm e} \stackrel
d \rightarrow N(0,1)$.

We now need to prove that the remaining terms $S_{\mathrm p}$,
$B_{\mathrm t}$ and $B_{\mathrm a}$, multiplied by
$N^{1/2}s_N^{-1}(b)$, are $o_p(1)$.  The term $S_{\mathrm p}$ is
treated in a similar manner as in the proof of Theorem
\ref{f_beta_rate}.
$$\left|S_{\mathrm p}(b)\right|
\le2\left(\frac1N\sum_{i=1}^N\frac{\left|\mathcal{H}^{-1}
\left(K_{2T_N}^{-}(x_i,\cdot)\right)(b)\right|}{\max(f_X(x_i),
m_N)}\right)\max_{i=1,\hdots,N}
\left|\frac{\max\left(f_X(x_i),m_N\right)}
{\max\left(\hat{f}_X^N(x_i),m_N\right)}-1\right|.$$
Using the Markov inequality, the empirical average in the
parenthesis is of the stochastic order of
$$m_N^{-1}\left\|\mathcal{H}^{-1}
\left(K_{2T_N}^{-}(\star,\cdot)\right)\right\|_1.
$$
But
\begin{align*}
m_N^{-1}\left\|\mathcal{H}^{-1}
\left(K_{2T_N}^{-}(\star,\cdot)\right)\right\|_1
&\leq
B(d,1)T_N^{d/2}m_N^{-1}\left\|
K_{2T_N}^{-}(\star,\cdot)\right\|_1\\
&\leq
B(d,1)T_N^{d/2}m_N^{-1}\left\|K_{2T_N}(\star,\cdot)\right\|_1
\end{align*}
where the first inequality follows from Theorem \ref{tbernstein} and
the second is obtained using the definition of the odd part and the
triangle inequality.   Note that the term
$\left\|K_{2T_N}(\star,\cdot)\right\|_1$
in the last line does not depend on $\cdot$ and is uniformly bounded.
By the lower bound \eqref{lyapounov1} it is enough to show
$N^{1/2}B(d,1)T_N^{-(d-1/2)}|S_{\mathrm p}(b)|=o_p(1)$.  From the inequality above,
$$N^{1/2}B(d,1)T_N^{-(d-1/2)}|S_{\mathrm p}(b)|\le\left(N^{1/2}T_N^{-(d-1)/2}
m_N^{-1}\right)\max_{i=1,\hdots,N}
\left|\frac{\max\left(f_X(x_i),m_N\right)}
{\max\left(\hat{f}_X(x_i),m_N\right)}-1\right|.$$
Its right hand side is of $o_p(1)$ if
$$\max_{i=1,\hdots,N}
\left|f_X(x_i)-\hat{f}_X(x_i)\right|
=o_p\left(N^{-1/2}T_N^{(d-1)/2}m_N^2\right),$$
which is met under \eqref{condPI}.

Let us now consider the bias term induced by the trimming procedure.
In the proof of Theorem \ref{f_beta_rate} we have obtained an upper bound
for $\|B_{\mathrm t}\|_{\infty}$ and we deduce that
$$N^{1/2}T_N^{-(d-1/2)}\|B_{\mathrm t}\|_{\infty}=o(1)$$
when condition \eqref{condapbias} is satisfied.
Finally, $N^{1/2}T_N^{-(d-1/2)}\|B_{\mathrm a}\|_{\infty}=o(1)$
if condition \eqref{edecay} is satisfied.
We conclude that the asymptotic normality holds for $b$ such that
$f_{\beta}(b)>0$.  The factor $4$ in the variance comes from the
fact that $\hat f_\beta = 2 \hat
f^-_\beta \hat \I$.
\end{proof}

The proof of Theorem \ref{t7b} is almost the same.

\end{document}